\documentclass[a4paper,10pt]{amsart}
\usepackage[english]{babel}
\usepackage[english]{translator}
\usepackage[T1]{fontenc}
\usepackage[utf8]{inputenc}
\usepackage{amsmath}
\usepackage{amssymb}
\usepackage{amsfonts}
\usepackage{amstext}
\usepackage{amsthm}
\usepackage{bbm}

\usepackage{graphicx}
\usepackage{color}
\usepackage{psfrag}
\usepackage{subfigure}
\usepackage{float}
\usepackage{marvosym}

\usepackage{hyperref}
\usepackage[
  hmarginratio={1:1},     % equal left and right margins
  vmarginratio={1:1},     % equal top and bottom margins
  textwidth=420pt,        % new text width
  heightrounded,          % always useful
]{geometry}

\hypersetup{
  pdffitwindow=false,
  pdfhighlight=/O,
  pdfnewwindow,
  colorlinks=true,
  citecolor=red,            % Farbe der Zitate(Buchverweise)
  linkcolor=blue,             % Farbe der Hyperref-Links
  menucolor=blue,            % color of Acrobat Reader menu buttons
  urlcolor=blue,             % Farbe der Internetadressverweise
  pdfpagemode=UseOutlines,
  bookmarksnumbered=true,
  linktocpage,
  pdftitle={The Identification Problem for complex-valued Ornstein-Uhlenbeck operators in $L^p(\mathbb{R}^d,\mathbb{C}^N)$},
  pdfsubject={complex-valued Ornstein-Uhlenbeck operator, identification problem in $L^p$, $L^p$-dissipativity, $L^p$-resolvent estimates, 
              maximal domain, applications to rotating waves},
  pdfauthor={Denny Otten},
  pdfkeywords={complex-valued Ornstein-Uhlenbeck operator, identification problem in $L^p$, $L^p$-dissipativity, $L^p$-resolvent estimates, 
               maximal domain, applications to rotating waves},
  pdfcreator={},
  pdfproducer={LaTeX mit hyperref}
}

\newcommand{\D}{\mathcal{D}}              % Definitionsbereich
\newcommand{\K}{\mathbb{K}}               % Koerper
\newcommand{\C}{\mathbb{C}}               % komplexe Zahlen
\newcommand{\R}{\mathbb{R}}               % reelle Zahlen
\newcommand{\N}{\mathbb{N}}                % natuerliche Zahlen
\newcommand{\F}{\mathcal{F}}              % Fourier Transformation
\renewcommand{\S}{\mathcal{S}}              % Schwartz Raum
\renewcommand{\Re}{\mathrm{Re}\,}          % Realteil
\newcommand{\trace}{\mathrm{tr}}           % Trace/Spur
\renewcommand{\L}{\mathcal{L}}             % L Operator
\newcommand{\one}{\mathbbm{1}}             % Indikatorfunktion
\newcommand{\Cb}{C_{\mathrm{b}}}

\newcommand{\amin}{a_{\mathrm{min}}} 
\newcommand{\amax}{a_{\mathrm{max}}} 
\newcommand{\azero}{a_0}
\newcommand{\aone}{a_1}
\newcommand{\atwo}{a_2} 
\newcommand{\bzero}{b_0}

\newcommand{\begriff}[1]{\textbf{#1}}

\makeatletter
  \def\section{\@startsection{section}{1}%
    \z@{.7\linespacing\@plus\linespacing}{.5\linespacing}%
    {\normalfont\LARGE\bfseries}}
\makeatother
\newcommand{\sect}
{
  \setcounter{equation}{0}
  \setcounter{figure}{0}
  \section
}
\newcommand{\enum}[1]{\textnormal{(\textbf{#1})}}

\theoremstyle{plain}
\newtheorem{definition}{Definition}[section]
\newtheorem{theorem}[definition]{Theorem}
\newtheorem{lemma}[definition]{Lemma}
\newtheorem{corollary}[definition]{Corollary}
\newtheorem{assumption}[definition]{Assumption}
\newtheorem{remark}[definition]{Remark}

\theoremstyle{definition}

\allowdisplaybreaks

\begin{document}
%  -----------
% |   Title   |
%  -----------
\title[The Identification Problem for complex-valued\\Ornstein-Uhlenbeck Operators in $L^p(\R^d,\C^N)$]{The Identification Problem for complex-valued\\Ornstein-Uhlenbeck Operators in $L^p(\R^d,\C^N)$}
%\maketitle
\setlength{\parindent}{0pt}
\vspace*{0.75cm}
\begin{center}
%\normalfont\LARGE\bfseries{\shorttitle}\\
\normalfont\huge\bfseries{\shorttitle}\\
\vspace*{0.25cm}
%\Large\bfseries\MakeUppercase{\shorttitle}
\end{center}

%  -------------
% |   Authors   |
%  -------------
\vspace*{0.5cm}
\noindent
\hspace*{4.2cm}
\textbf{Denny Otten}$\footnote[1]{e-mail: \textcolor{blue}{dotten@math.uni-bielefeld.de}, phone: \textcolor{blue}{+49 (0)521 106 4784}, \\
                                 fax: \textcolor{blue}{+49 (0)521 106 6498}, homepage: \url{http://www.math.uni-bielefeld.de/~dotten/}, \\
                                 supported by CRC 701 'Spectral Structures and Topological Methods in Mathematics'.}$ \\
\hspace*{4.2cm}
Department of Mathematics \\
\hspace*{4.2cm}
Bielefeld University \\
\hspace*{4.2cm}
33501 Bielefeld \\
\hspace*{4.2cm}
Germany

%  ----------
% |   Date   |
%  ----------
\vspace*{0.8cm}
\noindent
\hspace*{4.2cm}
Date: \today
\normalparindent=12pt

%  --------------
% |   Abstract   |
%  --------------
\vspace{0.4cm}
%\begin{abstract}
\noindent
\begin{center}
\begin{minipage}{0.9\textwidth}
  {\small
  \textbf{Abstract.} In this paper we study perturbed Ornstein-Uhlenbeck operators
  \begin{align*}
    \left[\mathcal{L}_{\infty} v\right](x)=A\triangle v(x) + \left\langle Sx,\nabla v(x)\right\rangle-B v(x),\,x\in\mathbb{R}^d,\,d\geqslant 2,
  \end{align*}
  for simultaneously diagonalizable matrices $A,B\in\mathbb{C}^{N,N}$. The unbounded drift term is defined by a skew-symmetric matrix
  $S\in\mathbb{R}^{d,d}$. Differential operators of this form appear when investigating rotating waves in time-dependent reaction
  diffusion systems. We prove under certain conditions that the maximal domain $\mathcal{D}(A_p)$ of the generator $A_p$ belonging to the
  Ornstein-Uhlenbeck semigroup coincides with the domain of $\mathcal{L}_{\infty}$ in $L^p(\mathbb{R}^d,\mathbb{C}^N)$ given by
  \begin{align*}
    \mathcal{D}^p_{\mathrm{loc}}(\mathcal{L}_0)=\left\{v\in W^{2,p}_{\mathrm{loc}}\cap L^p\mid A\triangle v + \left\langle S\cdot,\nabla v\right\rangle\in L^p\right\},\,1<p<\infty.
  \end{align*}
  One key assumption is a new $L^p$-dissipativity condition
  \begin{align*}
    |z|^2\Re\left\langle w,Aw\right\rangle + (p-2)\Re\left\langle w,z\right\rangle\Re\left\langle z,Aw\right\rangle\geqslant\gamma_A |z|^2|w|^2\;\forall\,z,w\in\C^N
  \end{align*}
  for some $\gamma_A>0$. The proof utilizes the following ingredients. First we show the closedness of $\mathcal{L}_{\infty}$ in $L^p$ and derive $L^p$-resolvent estimates 
  for $\mathcal{L}_{\infty}$. Then we prove that the Schwartz space is a core of $A_p$ and apply an $L^p$-solvability result of the resolvent equation for $A_p$.
  A second characterization shows that the maximal domain even coincides with
  \begin{align*}
    \mathcal{D}^p_{\mathrm{max}}(\mathcal{L}_0)=\{v\in W^{2,p}\mid \left\langle S\cdot,\nabla v\right\rangle\in L^p\},\,1<p<\infty.
  \end{align*}
  This second characterization is based on the first one, and its proof requires $L^p$-regularity for the Cauchy problem associated with $A_p$.
  Finally, we show a $W^{2,p}$-resolvent estimate for $\mathcal{L}_{\infty}$ and an $L^p$-estimate for the drift term $\left\langle S\cdot,\nabla v\right\rangle$.
  Our results may be considered as extensions of earlier works by Metafune, Pallara and Vespri to the vector-valued complex case.
  }
\end{minipage}
\end{center}
%\end{abstract}



%\tableofcontents

%---------------------------------------------------------------------------------------------------------------------------------------------------
%
%  SECTION 1: (Introduction)
%
%---------------------------------------------------------------------------------------------------------------------------------------------------
\sect{Introduction}
\label{sec:Introduction}
%---------------------------------------------------------------------------------------------------------------------------------------------------

% Introduction (-> Einleitung)
In this paper we study differential operators of the form
\begin{align*}
  \left[\L_{\infty}v\right](x) := A\triangle v(x) + \left\langle Sx,\nabla v(x)\right\rangle - Bv(x),\,x\in\R^d,\,d\geqslant 2,
\end{align*}
for simultaneously diagonalizable matrices $A,B\in\C^{N,N}$ with $\Re\sigma(A)>0$ and a skew-symmetric matrix $S\in\R^{d,d}$.

Introducing the complex Ornstein-Uhlenbeck operator, \cite{UhlenbeckOrnstein1930},
\begin{align*}
  \left[\L_0 v\right](x) := A\triangle v(x) + \left\langle Sx,\nabla v(x)\right\rangle,\,x\in\R^d,
\end{align*}
with \begriff{diffusion term} and \begriff{drift term} given by
\begin{align*}
  A\triangle v(x):=A\sum_{i=1}^{d}\frac{\partial^2}{\partial x_i^2}v(x)\quad\text{and}\quad\left\langle Sx,\nabla v(x)\right\rangle:=\sum_{i=1}^{d}(Sx)_i\frac{\partial}{\partial x_i}v(x),
\end{align*}
we observe that the operator $\L_{\infty}=\L_0-B$ is a constant coefficient perturbation of $\L_0$. Our interest is in skew-symmetric matrices $S=-S^T$, 
in which case the drift term is rotational containing angular derivatives
\begin{align*}
  \left\langle Sx,\nabla v(x)\right\rangle=\sum_{i=1}^{d-1}\sum_{j=i+1}^{d}S_{ij}\left(x_j\frac{\partial}{\partial x_i}-x_i\frac{\partial}{\partial x_j}\right)v(x).
\end{align*}
Such problems arise when investigating exponential decay of rotating waves in reaction diffusion systems, see \cite{Otten2014} and \cite{BeynLorenz2008}. 
The operator $\L_{\infty}$ appears as a far-field linearization at the solution of the nonlinear problem $\L_0 v=f(v)$. The results of this paper are 
crucial for dealing with the nonlinear case, see \cite{Otten2014}.

% Aim of the paper (roughly descibed) (-> Ziel grob formuliert)
The aim of this paper is to identify the maximal domain $\D(A_p)$ of the generator $A_p$ belonging to the perturbed Ornstein-Uhlenbeck semigroup in $L^p(\R^d,\C^N)$ 
for $1<p<\infty$. To be more precise, we prove in Theorem \ref{thm:LpMaximalDomainPart1} that the maximal domain $\D(A_p)$ in $L^p(\R^d,\C^N)$ coincides with the domain 
of $\L_{\infty}$ given by
\begin{align*}
  \D^p_{\mathrm{loc}}(\L_0)=\left\{v\in W^{2,p}_{\mathrm{loc}}(\R^d,\C^N)\cap L^p(\R^d,\C^N)\mid A\triangle v + \left\langle S\cdot,\nabla v\right\rangle\in L^p(\R^d,\C^N)\right\}
\end{align*}
for $1<p<\infty$. This result may be considered as an extension of \cite[Proposition 3.2]{Metafune2001}. Note that due to the smooth but unbounded coefficients 
in the drift term the semigroup is not analytic in $L^p(\R^d,\C^N)$ and the generator is not sectorial in $L^p(\R^d,\C^N)$. Therefore, the standard parabolic 
regularity results are not satisfied.

% Procedure of the paper (-> Vorgehensweise)
We first show that the Schwartz space is a core of the infinitesimal generator $(A_p,\D(A_p))$ in $L^p(\R^d,\C^N)$ for $1\leqslant p<\infty$. We then 
turn toward the perturbed Ornstein-Uhlenbeck operator $\L_{\infty}$ and prove the closedness of $\L_{\infty}$ on $\D^p_{\mathrm{loc}}(\L_0)$ in $L^p(\R^d,\C^N)$ 
for each $1<p<\infty$. Further, we derive $L^p$-resolvent estimates that imply uniqueness for solutions of the resolvent equation for $\L_{\infty}$ in $L^p(\R^d,\C^N)$ 
for $1<p<\infty$. Combining these three results with the (unique) solvability result of the resolvent equation for the generator $A_p$ from 
\cite[Corollary 5.5]{Otten2014a}, \cite[Corollary 6.7]{Otten2014}, we identify the maximal domain $\D(A_p)=\D^p_{\mathrm{loc}}(\L_0)$. The resolvent estimates 
for $\L_{\infty}$ require a new $L^p$-dissipativity condition for $\L_{\infty}$. For a more detailed outline we refer to Section \ref{sec:AssumptionsAndOutlineOfResults}.

% Literature (-> Literatur)
Identification problems concerning second order elliptic operators and in particular Ornstein-Uhlenbeck operators are treated in \cite{Metafune2001}, 
\cite{MetafunePallaraVespri2005} and \cite{PruessRhandiSchnaubelt2006} for $L^p$-spaces, in \cite{MetafunePruessRhandiSchnaubelt2002} for $L^p$-spaces with 
invariant measure and in \cite{DaPratoLunardi1995} for function spaces of bounded continuous and H\"older continuous functions. All these results are 
derived for the scalar real-valued case but for operators with more general principal part. $L^p$-dissipativity results of second order differential operators can be found 
in \cite{CialdeaMazya2005} and \cite{Cialdea2009}. 

% Applications (-> Anwendungen)

% Which results are adopted from the thesis
%The results from Section \ref{sec:TheLpAntieigenvalueCondition} are directly based on the PhD thesis \cite{Otten2014} while 
We emphasize that the results from Section \ref{subsec:ACoreForTheInfinitesimalGenerator}--\ref{sec:ExtensionsAndFurtherResults} are extensions of the results from 
the PhD thesis \cite{Otten2014} for arbitrary matrices $B\in\C^{N,N}$.

%  ----------------
% | Acknowledgment |
%  ----------------
%\noindent
\textbf{Acknowledgment.} The author is greatly indebted to Giorgio Metafune, Alessandra Lunardi and Wolf-J\"urgen Beyn for extensive discussions which helped in clarifying proofs.

%---------------------------------------------------------------------------------------------------------------------------------------------------
%
%  SECTION 2: (Assumptions and outline of results)
%
%---------------------------------------------------------------------------------------------------------------------------------------------------
\sect{Assumptions and outline of results}
\label{sec:AssumptionsAndOutlineOfResults}
%---------------------------------------------------------------------------------------------------------------------------------------------------

Consider the differential operator
\begin{align*}
  \left[\L_{\infty}v\right](x) := A\triangle v(x) + \left\langle Sx,\nabla v(x)\right\rangle - Bv(x),\,x\in\R^d,\,d\geqslant 2,
\end{align*}
for some matrices $A,B\in\C^{N,N}$ and $S\in\R^{d,d}$. 

The following conditions will be needed in this paper and relations among them will be discussed below.
%  ----------------
% | Assumption 2.1 |
%  ----------------
\begin{assumption} 
  \label{ass:Assumption1}
  Let $A,B\in\K^{N,N}$ with $\K\in\{\R,\C\}$ and $S\in\R^{d,d}$ be such that
  \begin{flalign}
    &\text{$A$ and $B$ are simultaneously diagonalizable (over $\C$)},         \tag{A1}\label{cond:A8B} &\\
    &\Re\sigma(A)>0, \tag{A2}\label{cond:A2} &\\
    &\text{There exists some $\beta_A>0$ such that} \tag{A3}\label{cond:A3} &\\
    &\quad\Re\left\langle w,Aw\right\rangle\geqslant\beta_A\;\forall\,w\in\K^N,\,|w|=1, \nonumber &\\
    &\text{There exists some $\gamma_A>0$ such that} \tag{A4}\label{cond:A4DC} &\\
    &\quad|z|^2\Re\left\langle w,Aw\right\rangle + (p-2)\Re\left\langle w,z\right\rangle\Re\left\langle z,Aw\right\rangle\geqslant\gamma_A |z|^2|w|^2\;\forall\,z,w\in\K^N \nonumber &\\
    &\text{for some $1<p<\infty$,} \nonumber &\\
    %&\text{Case ($N=1$, $\K=\R$):} \tag{A5}\label{cond:A4} \\
    %&\quad A=a>0 \nonumber &\\
    %&\text{Cases ($N\geqslant 2$, $\K=\R$) and ($N\geqslant 1$, $\K=\C$):} \nonumber &\\
    %&\quad\mu_1(A)>\frac{|p-2|}{p}\text{ for some $1<p<\infty$,} \nonumber &\\
    &\text{$S$ is skew-symmetric}.        \tag{A5}\label{cond:A5} &
  \end{flalign}
\end{assumption}

Assumption \eqref{cond:A8B} is a \textbf{system condition} and ensures that some results for scalar equations can be extended to system cases. This condition was used 
in \cite{Otten2014}, \cite{Otten2014a} to derive an explicit formula for the heat kernel of $\L_{\infty}$. It is motivated by the fact that a transformation of a scalar 
complex-valued equation into a $2$-dimensional real-valued system always implies two (real) matrices $A$ and $B$ that are simultaneously diagonalizable (over $\C$).
The \textbf{positivity condition} \eqref{cond:A2} guarantees that the diffusion part $A\triangle$ is an elliptic operator. It requires that all eigenvalues $\lambda$ of $A$ 
are contained in the open right half-plane $\C_+:=\{\lambda\in\C\mid\Re\lambda>0\}$, where $\sigma(A)$ denotes the spectrum of $A$. Condition \eqref{cond:A2} guarantees 
that $A^{-1}$ exists and states that $-A$ is a stable matrix. The \textbf{strong accretivity condition} \eqref{cond:A3} is more restrictive than \eqref{cond:A2}. In \eqref{cond:A3} 
$\left\langle u,v\right\rangle:=\overline{u}^T v$ denotes the standard inner product on $\K^N$. Note that the condition \eqref{cond:A2} is satisfied if and only if
\begin{align*}
  \exists\,\left[\cdot,\cdot\right]\text{ inner product on $\K^N$}:\quad\Re\left[w,Aw\right]\geqslant\beta_A>0\;\forall\,w\in\K^N,\,\left[w,w\right]=1,
\end{align*}
but it does not imply $\left[\cdot,\cdot\right]=\left\langle\cdot,\cdot\right\rangle$. Condition \eqref{cond:A3} ensures that the differential operator $\L_{\infty}$ 
is closed on its (local) domain $\D^p_{\mathrm{loc}}(\L_0)$. The \textbf{$L^p$-dissipativity condition} \eqref{cond:A4DC} seems to be new in the literature and is used 
to prove $L^p$-resolvent estimates for $\L_{\infty}$. Condition \eqref{cond:A4DC} is more restrictive than \eqref{cond:A3} and imposes additional requirements on the 
spectrum of $A$. A geometrical meaning of \eqref{cond:A4DC} in terms of the antieigenvalues of the diffusion matrix $A$ can be found in \cite[Theorem 5.18]{Otten2014}.
We summarize the following relation of assumptions \eqref{cond:A2}--\eqref{cond:A4DC}:
\begin{align*}
  \text{\eqref{cond:A2}}\Longleftarrow\text{\eqref{cond:A3}}\Longleftarrow\text{\eqref{cond:A4DC}}.
\end{align*}
The \textbf{rotational condition} \eqref{cond:A5} implies that the drift term contains only angular derivatives, which is crucial for use our results from \cite{Otten2014a}.

For a matrix $C\in\K^{N,N}$ we denote by $\sigma(C)$ the \begriff{spectrum of $C$}, by $\rho(C):=\max_{\lambda\in\sigma(C)}\left|\lambda\right|$ the \begriff{spectral radius of $C$} 
and by $s(C):=\max_{\lambda\in\sigma(C)}\Re\lambda$ the \begriff{spectral abscissa} (or \begriff{spectral bound}) \begriff{of $C$}. Using this notation, we define the constants
\begin{equation}
  \begin{aligned}
  \amin :=& \left(\rho\left(A^{-1}\right)\right)^{-1}, &&\amax  := \rho(A),              &&\azero := -s(-A),\\
  \aone :=& \frac{\amax^2}{\amin\azero},               &&\atwo:=\frac{4\amax^2}{\azero}, &&\bzero := -s(-B).
  \end{aligned}
  \label{equ:aminamaxazerobzero}
\end{equation}
These constants appear in \cite{Otten2014}, \cite{Otten2014a}. Moreover, let $\beta_B\in\R$ be such that
\begin{align}
  \label{equ:betaB}
  \Re\left\langle w,Bw\right\rangle\geqslant -\beta_B\;\forall\,w\in\K^N,\,|w|=1.
\end{align}
If $\beta_B\leqslant 0$, \eqref{equ:betaB} can be considered as a \textbf{dissipativity condition} for $-B$.

% Definition of function spaces
We introduce Lebesgue and Sobolev spaces via
\begin{align*}
  L^p(\R^d,\K^N) :=& \left\{v\in L^1_{\mathrm{loc}}(\R^d,\K^N)\mid \left\|v\right\|_{L^p}<\infty\right\}, \\
  W^{k,p}(\R^d,\K^N) :=& \left\{v\in L^p(\R^d,\K^N)\mid D^{\beta}v\in L^p(\R^d,\K^N)\;\forall\,|\beta|\leqslant k\right\},
\end{align*}
with norms
\begin{align*}
  \left\|v\right\|_{L^p(\R^d,\K^N)} :=& \bigg(\int_{\R^d}\left|v(x)\right|^p dx\bigg)^{\frac{1}{p}}, \\
  \left\|v\right\|_{W^{k,p}(\R^d,\K^N)} :=& \bigg(\sum_{0\leqslant|\beta|\leqslant k}\left\|D^{\beta}v\right\|_{L^p(\R^d,\K^N)}^p\bigg)^{\frac{1}{p}},
\end{align*}
for every $1\leqslant p<\infty$, $k\in\N_0$ and multiindex $\beta\in\N_0^d$. Moreover, we define the \begriff{Schwartz space} via, \cite[VI.5.1 Definition]{EngelNagel2000},
\begin{align}
  \label{equ:RapidlyDecreasingProperty}
  \S(\R^d,\K^N):=\left\{\phi\in C^{\infty}(\R^d,\K^N)\mid\lim_{\left|x\right|\rightarrow\infty} x^{\alpha}D^{\beta}\phi(x)=0\;\forall\,\alpha,\beta\in\N_0^d\right\}.
\end{align}
which we sometimes abbreviate by $\S$. Endowed with the family of seminorms
\begin{align*}
  \left|\phi\right|_{\alpha,\beta}:=\sup_{x\in\R^d}\left| x^{\alpha} D^{\beta}\phi(x)\right|
\end{align*}
the Schwartz space $\S$ becomes a Fr\'{e}chet space containing $C^{\infty}_{\mathrm{c}}(\R^d,\K^N)$ as a dense subspace.

% What has been shown before
Before we give a detailed outline we briefly review and collect some results from \cite{Otten2014} and \cite{Otten2014a} used in this paper.

Assuming \eqref{cond:A8B}, \eqref{cond:A2} and \eqref{cond:A5} for $\K=\C$ it is shown in \cite[Theorem 4.2-4.4]{Otten2014}, \cite[Theorem 3.1]{Otten2014a} 
that the function $H:\R^d\times\R^d\times]0,\infty[\rightarrow\C^{N,N}$ defined by
\begin{align}
  \label{equ:HeatKernel}
  H(x,\xi,t)=(4\pi t A)^{-\frac{d}{2}}\exp\left(-Bt-(4tA)^{-1}\left|e^{tS}x-\xi\right|^2\right)
\end{align}
is a \textbf{heat kernel} of the perturbed Ornstein-Uhlenbeck operator
\begin{align}
  \label{equ:Linfty2}
  \left[\L_{\infty}v\right](x):=A\triangle v(x)+\left\langle Sx,\nabla v(x)\right\rangle-Bv(x).
\end{align}
This means, that $H$ satisfies the following \begriff{heat kernel properties}
\begin{align}
  &H\in C^{2,2,1}(\R^d\times\R^d\times\R_+^*,\C^{N,N}),                                          \tag{H1}\label{equ:H1} \\
  &\frac{\partial}{\partial t}H(x,\xi,t)=\L_{\infty}H(x,\xi,t) &&\forall\,x,\xi\in\R^d,\,t>0,      \tag{H2}\label{equ:H2} \\
  &\lim_{t\downarrow 0}H(x,\xi,t)=\delta_{x}(\xi)I_N             &&\forall\,x,\xi\in\R^d.          \tag{H3}\label{equ:H3}
\end{align}
Under the same assumptions it is proved in \cite[Theorem 5.3]{Otten2014a} that the family of mappings 
$T(t):L^p(\R^d,\C^N)\rightarrow L^p(\R^d,\C^N)$, $t\geqslant 0$, defined by
\begin{align}
  \left[T(t)v\right](x):= \begin{cases}
                              \int_{\R^d}H(x,\xi,t)v(\xi)d\xi &\text{, }t>0 \\
                              v(x) &\text{, }t=0
                            \end{cases}\quad ,x\in\R^d,
  \label{equ:OrnsteinUhlenbeckSemigroupLp}
\end{align}
generates a strongly continuous semigroup on $L^p(\R^d,\C^N)$ for each $1\leqslant p<\infty$. The semigroup $\left(T(t)\right)_{t\geqslant 0}$ 
is called the Ornstein-Uhlenbeck semigroup if $B=0$. The strong continuity of the semigroup justifies to introduce the infinitesimal generator 
$A_p:L^p(\R^d,\C^N)\supseteq\D(A_p)\rightarrow L^p(\R^d,\C^N)$ of $\left(T(t)\right)_{t\geqslant 0}$, short $\left(A_p,\D(A_p)\right)$, via
\begin{align*}
  A_p v := \lim_{t\downarrow 0}\frac{T(t)v-v}{t},\; 1\leqslant p<\infty
\end{align*}
for every $v\in\D(A_p)$, where the domain (or maximal domain) of $A_p$ is given by
\begin{align*}
  \D(A_p):=&\left\{v\in L^p(\R^d,\C^N)\mid \lim_{t\downarrow 0}\frac{T(t)v-v}{t}\text{ exists in $L^p(\R^d,\C^N)$}\right\} \\
          =&\left\{v\in L^p(\R^d,\C^N)\mid A_p v\in L^p(\R^d,\C^N)\right\}.
\end{align*}
An application of abstract semigroup theory yields the following Corollary \ref{cor:OrnsteinUhlenbeckLpSolvabilityUniqueness} concerning the unique solvability 
of the resolvent equation for $A_p$ in $L^p(\R^d,\C^N)$ for $1\leqslant p<\infty$. This is an essential component for the proof of our main result in 
Theorem \ref{thm:LpMaximalDomainPart1} and is proved in \cite[Corollary 5.5]{Otten2014a}, \cite[Corollary 6.7]{Otten2014}.

%  ---------------
% | Corollary 2.2 |
%  ---------------
\begin{corollary}[Solvability and uniqueness in $L^p(\R^d,\C^N)$]\label{cor:OrnsteinUhlenbeckLpSolvabilityUniqueness}
  Let the assumptions \eqref{cond:A8B}, \eqref{cond:A2} and \eqref{cond:A5} be satisfied for $\K=\C$ and let $1\leqslant p<\infty$. 
  Moreover, let $\lambda\in\C$ with $\Re\lambda>-\bzero$, where $\bzero$ is from \eqref{equ:aminamaxazerobzero}. 
  Then for every $g\in L^p(\R^d,\C^N)$ the resolvent equation
  \begin{align*}
    \left(\lambda I-A_p\right)v = g
  \end{align*}
  admits a unique solution $v_{\star}\in\D(A_p)$, which is given by the integral expression
  \begin{align}
    v_{\star} = R(\lambda,A_p)g = \int_{0}^{\infty}e^{-\lambda t}T(t)g dt 
                            = \int_{0}^{\infty}e^{-\lambda t}\int_{\R^d}H(\cdot,\xi,t)g(\xi)d\xi dt.
    \label{equ:IntegralExpressionLp}
  \end{align}
  Moreover, the following resolvent estimate holds
  \begin{align*}
    \left\|v_{\star}\right\|_{L^p(\R^d,\C^N)}\leqslant \frac{a_1^{\frac{d}{2}}}{\Re\lambda+\bzero}\left\|g\right\|_{L^p(\R^d,\C^N)}.
  \end{align*}
\end{corollary}

% Detailed outline
In Section \ref{subsec:ACoreForTheInfinitesimalGenerator} we analyze subspaces of the maximal domain $\D(A_p)$. Assuming 
\eqref{cond:A8B}, \eqref{cond:A2} and \eqref{cond:A5} for $\K=\C$ we prove in Theorem \ref{thm:CoreForTheInfinitesimalGenerator} 
that the Schwartz space $\S(\R^d,\C^N)$ is a core of the infinitesimal generator $\left(A_p,\D(A_p)\right)$ for every 
$1\leqslant p<\infty$. In particular, the \begriff{abstract operator} $A_p$ coincides with the \begriff{formal operator} $\L_{\infty}$ 
on $\S(\R^d,\C^N)$ for $1\leqslant p<\infty$, meaning that $A_p\phi=\L_{\infty}\phi$ for every $\phi\in\S(\R^d,\C^N)$. 
The proof uses Lebesgue's dominated convergence theorem in a similar way to \cite[Proposition 2.2 and 3.2]{Metafune2001} 
and ideas from \cite[II.2.13]{EngelNagel2000}.

In Section \ref{sec:TheOrnsteinUhlenbeckOperatorInLp} we analyze the operator $\L_{\infty}:L^p(\R^d,\C^N)\supseteq\D^p_{\mathrm{loc}}(\L_0)\rightarrow L^p(\R^d,\C^N)$ 
on its domain
\begin{align*}
  \D^p_{\mathrm{loc}}(\L_0):=\left\{v\in W^{2,p}_{\mathrm{loc}}(\R^d,\C^N)\cap L^p(\R^d,\C^N)\mid A\triangle v+\left\langle S\cdot,\nabla v\right\rangle\in L^p(\R^d,\C^N)\right\}.
\end{align*}
Under the assumption \eqref{cond:A3} for $\K=\C$ we show in Lemma \ref{lem:ClosednessOfL0} that $\L_{\infty}$ is a closed 
operator in $L^p(\R^d,\C^N)$ for every $1<p<\infty$. This justifies to introduce and analyze the resolvent of $\L_{\infty}$. 
Assuming stonger assumption \eqref{cond:A4DC} and \eqref{cond:A5} for $1<p<\infty$ and $\K=\C$, we prove in Theorem \ref{thm:UniquenessInDpmax} 
that solutions $v_{\star}\in\D^p_{\mathrm{loc}}(\L_0)$ of the resolvent equation
\begin{align*}
  \left(\lambda I-\L_{\infty}\right)v=g
\end{align*}
are unique for every $g\in L^p(\R^d,\C^N)$ and $\lambda\in\C$ with $\Re\lambda>\beta_B$, where $\beta_B$ is from \eqref{equ:betaB}. 
The main idea of the proof comes from \cite[Theorem 2.2 and Remark 2.3]{MetafunePallaraVespri2005} for the scalar real-valued case. 
But we refer also to \cite[Theorem 3.1]{BeynLorenz2008} for the special case $d=2$ with positive definite matrix $A\in\R^{N,N}$. 
In contrast to \cite{MetafunePallaraVespri2005} and \cite{BeynLorenz2008}, our proof requires the additional $L^p$-dissipativity 
condition \eqref{cond:A4DC}.

In Section \ref{sec:IdentificationProblemForTheComplexOrnsteinUhlenbeckOperatorInLp} we solve the identification problem for the 
perturbed Ornstein-Uhlenbeck operator $\L_{\infty}$. Assuming \eqref{cond:A8B}, \eqref{cond:A4DC} and \eqref{cond:A5} for $1<p<\infty$ 
and $\K=\C$, we prove in Theorem \ref{thm:LpMaximalDomainPart1} that the maximal domain $\D(A_p)$ equals $\D^p_{\mathrm{loc}}(\L_0)$. 
In particular, we show that the abstract operator $A_p$ and the formal operator $\L_{\infty}$ coincide on $\D(A_p)$. The proof of 
Theorem \ref{thm:LpMaximalDomainPart1} is structured as follows: To prove $\D(A_p)\subseteq\D^p_{\mathrm{loc}}(\L_0)$ we need 
that $\S(\R^d,\C^N)$ is a core of $(A_p,\D(A_p))$ (Theorem \ref{thm:CoreForTheInfinitesimalGenerator}) and the closedness of $\L_{\infty}$ 
(Lemma \ref{lem:ClosednessOfL0}). Conversely, the inclusion $\D(A_p)\supseteq\D^p_{\mathrm{loc}}(\L_0)$ requires the (unique) solvability of the 
resolvent equation for $A_p$ (Corollary \ref{cor:OrnsteinUhlenbeckLpSolvabilityUniqueness}), the uniqueness for solutions of the resolvent equation 
for $\L_{\infty}$ (Theorem \ref{thm:UniquenessInDpmax}) and the inclusion $\D(A_p)\subseteq\D^p_{\mathrm{loc}}(\L_0)$ that has been shown before.
The main idea for the first part of the proof comes from \cite[Proposition 2.2 and 3.2]{Metafune2001}, where such a result was proved 
for the scalar real-valued Ornstein-Uhlenbeck operator
\begin{align*}
  \trace(Q D^2 v(x))+\left\langle Sx,\nabla v(x)\right\rangle,\,x\in\R^d
\end{align*}
with $Q\in\R^{d,d}$, $Q=Q^T$, $Q>0$ and $0\neq S\in\R^{d.d}$. We conclude this section with some extensions and further results concerning 
the characterization of the maximal domain.

In Section \ref{sec:ExtensionsAndFurtherResults} we present a second characterization of the maximal domain $\D(A_p)$ of $A_p$. Assuming 
\eqref{cond:A8B}, \eqref{cond:A4DC} and \eqref{cond:A5} for $1<p<\infty$ and $\K=\C$, we prove in Theorem \ref{equ:MaximalDomainPart2} that 
the maximal domain $\D(A_p)$ even coincides with
\begin{align*}
  \D^p_{\mathrm{max}}(\L_0):=\left\{v\in W^{2,p}(\R^d,\C^N)\mid \left\langle S\cdot,\nabla v\right\rangle\in L^p(\R^d,\C^N)\right\}.
\end{align*}
The proof of $\D(A_p)\supseteq\D^p_{\mathrm{max}}(\L_0)$ is straightforward. To prove $\D(A_p)\subseteq\D^p_{\mathrm{max}}(\L_0)$ we analyze 
the abstract Cauchy problem for $A_p$ and apply $L^p$-regularity results from \cite[Theorem 5.1]{Otten2014a}, \cite[Theorem 6.3]{Otten2014} 
for the homogeneous, and from \cite[Theorem 5.24]{Otten2014} for the inhomogeneous Cauchy problem. The main idea of the proof comes from 
\cite{MetafunePallaraVespri2005}, where this result is proved for the scalar real-valued Ornstein-Uhlenbeck operator. We emphasize that the 
proof of \cite[Theorem 5.24]{Otten2014} uses a generalization of \cite[IV. Theorem 9.1]{LadyzenskajaSolonnikovUralceva1968} to the complex-valued 
case (which, however, has not been carried out in detail). Finally, under the same assumptions, we prove in Corollary \ref{thm:ResolventEstimatesForL0InLp} 
a $W^{2,p}$-resolvent estimate for $A_p$ and an $L^p$-estimate for the drift term $\left\langle S\cdot,\nabla v\right\rangle$. The proof 
utilizes the norm equivalence of the graph norms of $A_p$ and $\L_{\infty}$ from \cite[Corollary 5.26]{Otten2014}.

%---------------------------------------------------------------------------------------------------------------------------------------------------
%
%  SECTION 3: (A core for the infinitesimal generator)
%
%---------------------------------------------------------------------------------------------------------------------------------------------------
\sect{A core for abstract Ornstein-Uhlenbeck operators in \texorpdfstring{$L^p(\R^d,\C^N)$}{Lp(Rd,CN)}}
\label{subsec:ACoreForTheInfinitesimalGenerator}
%---------------------------------------------------------------------------------------------------------------------------------------------------

The aim of this section is to show that the Schwartz space $\S:=\S(\R^d,\C^N)$ is a core for the infinitesimal generator $(A_p,\D(A_p))$ of $\L_{\infty}$ 
for every $1\leqslant p<\infty$. For the following definition see \cite[II.1.6 Definition]{EngelNagel2000}.

%  ----------------
% | Definition 3.1 |
%  ----------------
\begin{definition}
  A subspace $D\subseteq\D(A_p)$ of the maximal domain $\D(A_p)$ of the linear operator $A_p:L^p(\R^d,\C^N)\supseteq\D(A_p)\rightarrow L^p(\R^d,\C^N)$ 
  with $1\leqslant p<\infty$ is called a \begriff{core for $\left(A_p,\D(A_p)\right)$} if $D$ is dense in $\D(A_p)$ with respect to the graph norm of $A_p$
  \begin{align*}
    \left\|v\right\|_{A_p}:=\left\|A_p v\right\|_{L^p(\R^d,\C^N)}+\left\|v\right\|_{L^p(\R^d,\C^N)},\quad v\in\D(A_p).
  \end{align*}
\end{definition}

The next theorem states that the Schwartz space $\S(\R^d,\C^N)$ is a core for the infinitesimal generator $(A_p,\D(A_p))$ of the semigroup $\left(T(t)\right)_{t\geqslant 0}$ 
in $L^p(\R^d,\C^N)$ for $1\leqslant p<\infty$. Moreover, it turns out that the formal operator $\L_{\infty}$ and the abstract operator $A_p$ coincide on the 
Schwartz space $\S(\R^d,\C^N)$. This is an extension of the real-valued scalar result in \cite[Proposition 2.2 and 3.2]{Metafune2001} to complex valued systems 
and also an extension of \cite[Theorem 5.10]{Otten2014} which assumes $B=0$.

%  -------------
% | Theorem 3.2 |
%  -------------
\begin{theorem}[Core for the infinitesimal generator]\label{thm:CoreForTheInfinitesimalGenerator}
  Let the assumptions \eqref{cond:A8B}, \eqref{cond:A2} and \eqref{cond:A5} be satisfied for $\K=\C$ and let $1\leqslant p<\infty$. Then the Schwartz space 
  $\S(\R^d,\C^N)\subseteq\D(A_p)$ is a core for $\left(A_p,\D(A_p)\right)$.
\end{theorem}

%  ---------------------
% | Proof (Theorem 3.2) |
%  ---------------------
\begin{proof}
  The proof is subdivided into the following three steps:
  \begin{enumerate}
    \item[\enum{1}] $\S\subseteq L^p(\R^d,\C^N)$ is dense w.r.t. $\left\|\cdot\right\|_{L^p(\R^d,\C^N)}$,
    \item[\enum{2}] $\S$ is a subspace of $\D(A_p)$, i.e. $\S\subseteq\D(A_p)$ and $A_p\phi=\L_{\infty}\phi$ $\forall\,\phi\in\S$,
    \item[\enum{3}] $\S$ is invariant under the semigroup $\left(T(t)\right)_{t\geqslant 0}$, i.e. $T(t)\S\subseteq\S$ $\forall\,t\geqslant 0$.
  \end{enumerate}
  The assertion of Theorem \ref{thm:CoreForTheInfinitesimalGenerator} then follows directly from an application of \cite[II.1.7 Proposition]{EngelNagel2000}.\\
  \enum{1}: Due to the inclusion $C^{\infty}_{\mathrm{c}}(\R^d,\C^N)\subseteq\S(\R^d,\C^N)\subseteq L^p(\R^d,\C^N)$ and since $C^{\infty}_{\mathrm{c}}(\R^d,\C^N)$ 
  is dense in $L^p(\R^d,\C^N)$ w.r.t. $\left\|\cdot\right\|_{L^p}$ for every $1\leqslant p<\infty$, we deduce that $\S(\R^d,\C^N)$ is also dense in $L^p(\R^d,\C^N)$ 
  w.r.t. $\left\|\cdot\right\|_{L^p}$ for every $1\leqslant p<\infty$. \\
  \enum{2}: Let $\phi\in\S(\R^d,\C^N)$ be arbitrary. In order to prove $\S\subseteq\D(A_p)$ we must show that
  \begin{align*}
    \phi\in L^p(\R^d,\C^N),\;
    \L_{\infty}\phi\in L^p(\R^d,\C^N),\;
    \lim_{t\downarrow 0}\frac{1}{t}\left(T(t)\phi-\phi\right)\text{ exists in }L^p(\R^d,\C^N).
  \end{align*}
  1. Since $\S(\R^d,\C^N)$ is a subspace of $L^p(\R^d,\C^N)$ for $1\leqslant p<\infty$, we deduce\linebreak $\phi\in L^p(\R^d,\C^N)$. Therefore, it is sufficient 
  to show $\L_{\infty}\phi\in\S(\R^d,\C^N)$. Then we deduce $\L_{\infty}\phi\in L^p(\R^d,\C^N)$ by the same argument. Since $\phi\in\S(\R^d,\C^N)\subseteq C^{\infty}(\R^d,\C^N)$ 
  and since $\L_{\infty}$ has smooth coefficients we infer that $\L_{\infty}\phi\in C^{\infty}(\R^d,\C^N)$. Considering the operator
  \begin{align*}
       \left[\L_{\infty}\phi\right](x)
    =& A\triangle\phi(x)+\left\langle Sx,\nabla\phi(x)\right\rangle-B\phi(x) \\
    =& A\sum_{i=1}^{d}D_i^2\phi(x) + \sum_{i=1}^{d}\sum_{j=1}^{d}S_{ij}x_j D_i\phi(x)-B\phi(x)
  \end{align*}
  and 
  \begin{align*}
       x^{\tilde{\alpha}}D^{\tilde{\beta}}\left[\L_0\phi\right](x)
    = A\sum_{i=1}^{d}x^{\tilde{\alpha}}D^{\tilde{\beta}}D_i^2\phi(x) +\sum_{i=1}^{d}\sum_{j=1}^{d}S_{ij}x_jx^{\tilde{\alpha}}D^{\tilde{\beta}}D_i\phi(x)
      -Bx^{\tilde{\alpha}}D^{\tilde{\beta}}\phi(x)
  \end{align*}
  for $\tilde{\alpha},\tilde{\beta}\in\N_0^d$ and using the fact that $\phi$ is rapidly decreasing, we conclude from \eqref{equ:RapidlyDecreasingProperty} with 
  $\alpha=\tilde{\alpha},\,\beta=\tilde{\beta}+2e_i$ and $\alpha=\tilde{\alpha}+e_j,\,\beta=\tilde{\beta}+e_i$, that every term on the right hand side vanishes 
  as $|x|$ goes to infinity. Hence, $\L_{\infty}\phi\in\S$. It remains to verify that the limit exists in $L^p(\R^d,\C^N)$. \\
  2. We first give a motivation how the limit looks like: Using the heat kernel properties \eqref{equ:H2} and \eqref{equ:H3} a formal computation shows
  \begin{align*}
    \left[A_p\phi\right](x) 
    :=& \lim_{t\downarrow 0}\frac{\left[T(t)\phi\right](x)-\phi(x)}{t}
    =   \lim_{t\downarrow 0}\left[\frac{T(t)-T(0)}{t-0}\right]\phi(x) \\
    =&  \left[\frac{d}{dt}\left[T(t)\phi\right](x)\right]_{t=0}
    =   \left[\int_{\R^d}\frac{\partial}{\partial t}H(x,\xi,t)\phi(\xi)d\xi\right]_{t=0}\\
    =&  \left[\L_{\infty}\int_{\R^d}H(x,\xi,t)\phi(\xi)d\xi\right]_{t=0}
    =   \L_{\infty}\int_{\R^d}\delta_{x}(\xi)\phi(\xi)d\xi
    =   \left[\L_{\infty}\phi\right](x).
  \end{align*}
  This suggests that the limit tends (pointwise) to $A_p\phi(x):=\L_{\infty}\phi(x)\in L^p(\R^d,\C^N)$, provided that all steps in the calculation are justified. We next 
  prove that the limit even exists in $L^p(\R^d,\C^N)$ w.r.t. $\left\|\cdot\right\|_{L^p}$, which is indeed much more involved, \cite[Proposition 2.2 and 3.2]{Metafune2001}. \\ 
  3. Our aim is to apply Lebesgue's dominated convergence theorem in $L^p$ from \cite[Satz 1.23]{Alt2006} with 
  \begin{align*}
    f_t(x):=&\frac{\left[T(t)\phi\right](x)-\phi(x)}{t},\quad f(x):=\left[\L_{\infty}\phi\right](x)
  \end{align*}
  to deduce that $f_t,f\in L^p(\R^d,\C^N)$ for $t>0$ and $f_t\rightarrow f$ in $L^p(\R^d,\C^N)$ as $t\downarrow 0$. We then directly conclude $\phi\in\D(A_p)$, thus 
  $\S(\R^d,\C^N)\subseteq\D(A_p)$. In particular, we have $A_p\phi:=\L_{\infty}\phi$ for every $\phi\in\S(\R^d,\C^N)$. To justify the application of dominated convergence 
  we must show that
  \begin{enumerate}
    \item[(a)] $f_t(x)\rightarrow f(x)$ pointwise for a.e. $x\in\R^d$ as $t\downarrow 0$,
    \item[(b)] $\left|f_t(x)\right|\leqslant g(x)$ pointwise for a.e. $x\in\R^d$ and for every $0<t\leqslant t_0$,
    \item[(c)] $g\in L^p(\R^d,\R)$,
  \end{enumerate}
  where the function $g$ is constructed during the proof. Before we start to verify the properties (a)--(c) we simplify the term $f_t$, \cite[Proposition 2.2 and 3.2]{Metafune2001}: 
  Since $\phi\in\S(\R^d,\C^N)$, Taylor's formula up to order $2$ yields
  \begin{align*}
    \phi(e^{tS}x-\psi) =& \phi(x) + \sum_{i=1}^{d}\left(e^{tS}x-x-\psi\right)_i D_i\phi(x) \\
                        & +\sum_{i=1}^{d}\sum_{j=1}^{d}\frac{1}{2}\left(e^{tS}x-x-\psi\right)_i\left(e^{tS}x-x-\psi\right)_j D_jD_i\phi(x) \\
                        & +R_{x,2}\left(e^{tS}x-x-\psi\right)
  \end{align*}
  with remainder
  \begin{align*}
    R_{x,2}\left(z-x\right)=\sum_{|\beta|=3}\frac{|\beta|}{\beta!}\left(z-x\right)^{\beta}\int_{0}^{1}(1-\tau)^{|\beta|-1} D^{\beta}\phi\left(x+\tau\left(z-x\right)\right) d\tau 
  \end{align*}
  for $z:=e^{tS}x-\psi$ satisfying 
  \begin{align}
    \label{equ:EstimateRemainderTaylorExpansion}
    \left|R_{x,2}\left(z-x\right)\right|\leqslant C_{\beta}C_{\phi}\left|z-x\right|^3,
  \end{align}
  where $C_{\beta}:=\sum_{|\beta|=3}\frac{1}{\beta!}$ and $C_{\phi}:=\max_{|\beta|=3}\sup_{y\in\R^d}\left|D^{\beta}\phi(y)\right|$. Thus, using 
  \eqref{equ:OrnsteinUhlenbeckSemigroupLp}, the transformation theorem (with transformation $\Phi(\xi)=e^{tS}x-\xi$) and
  \begin{align*}
    K(\psi,t):=H(x,e^{-tS}x-\psi,t)=\left(4\pi t A\right)^{-\frac{d}{2}}\exp\left(-Bt-\left(4tA\right)^{-1}|\psi|^2\right),\,t>0
  \end{align*}
  we obtain
  \begin{align*}
        f_t(x) 
    :=& \frac{\left[T(t)\phi\right](x)-\phi(x)}{t} 
     =  \frac{1}{t}\left[\int_{\R^d}H(x,\xi,t)\phi(\xi)d\xi-\phi(x)\right] \\
     =& \frac{1}{t}\left[\int_{\R^d}K(\psi,t)\phi\left(e^{tS}x-\psi\right)d\psi -\phi(x)\right] \\
     =& \frac{1}{t}\left[\int_{\R^d}K(\psi,t)d\psi-I_N\right]\phi(x)
       +\frac{1}{t}\int_{\R^d}K(\psi,t)\sum_{i=1}^{d}\left(e^{tS}x-x-\psi\right)_i D_i\phi(x) d\psi \\ 
     & +\frac{1}{t}\int_{\R^d}K(\psi,t)\sum_{i=1}^{d}\sum_{j=1}^{d}\frac{1}{2}\left(e^{tS}x-x-\psi\right)_i\left(e^{tS}x-x-\psi\right)_j D_jD_i\phi(x) d\psi \\
     & +\frac{1}{t}\int_{\R^d}K(\psi,t) R_{x,2}\left(e^{tS}x-x-\psi\right) d\psi =: \sum_{i=1}^{4}T_i(x,t),\;t>0.
  \end{align*}
  $T_1$: Using \cite[Lemma 4.5(1)]{Otten2014a} we obtain for every $t>0$
  \begin{align*}
    T_1(x,t) = \frac{1}{t}\left[\int_{\R^d}K(\psi,t)d\psi-I_N\right]\phi(x) = \left(\frac{e^{-Bt}-I_N}{t}\right)\phi(x).
  \end{align*}
  $T_2$: A decomposition of $T_2$ leads to
  \begin{align*}
     & T_2(x,t) = \frac{1}{t}\int_{\R^d}K(\psi,t)\sum_{i=1}^{d}\left(e^{tS}x-x-\psi\right)_i D_i\phi(x)  d\psi \\
    =& \frac{1}{t}\int_{\R^d}K(\psi,t)d\psi\sum_{i=1}^{d}\left(e^{tS}x-x\right)_i D_i\phi(x) 
       -\frac{1}{t}\sum_{i=1}^{d}\int_{\R^d}K(\psi,t)\psi_i d\psi D_i\phi(x) \\
    =& e^{-Bt}\sum_{i=1}^{d}\left(\frac{e^{tS}x-x}{t}\right)_i D_i\phi(x)
  \end{align*}
  for every $t>0$, where we used \cite[Lemma 4.5(1) and (2)]{Otten2014a} for the first and second term, respectively. \\
  $T_3$: Similarly, a decomposition of $T_3$ leads to
  \begin{align*}
     & T_3(x,t) = \frac{1}{t}\int_{\R^d}K(\psi,t)\sum_{i=1}^{d}\sum_{j=1}^{d}\frac{1}{2} \left(e^{tS}x-x-\psi\right)_i\left(e^{tS}x-x-\psi\right)_j D_jD_i\phi(x) d\psi \\
    =& \frac{1}{2t}\sum_{i=1}^{d}\sum_{j=1}^{d}\int_{\R^d}K(\psi,t)\psi_i\psi_j d\psi D_jD_i\phi(x) \\
     & +\frac{1}{2t}\sum_{i=1}^{d}\sum_{j=1}^{d}\int_{\R^d}K(\psi,t) d\psi \left(e^{tS}x-x\right)_i\left(e^{tS}x-x\right)_j D_jD_i\phi(x) \\
     & -\frac{1}{2t}\sum_{i=1}^{d}\sum_{j=1}^{d}\int_{\R^d}K(\psi,t)\left[\left(e^{tS}x-x\right)_i\psi_j+\left(e^{tS}x-x\right)_j\psi_i\right]d\psi D_jD_i\phi(x) \\
    =& \frac{1}{2t}\sum_{i=1}^{d}2te^{-Bt}AD_i^2\phi(x) 
       +\frac{t}{2}e^{-Bt}\sum_{i=1}^{d}\sum_{j=1}^{d}\left(\frac{e^{tS}x-x}{t}\right)_i\left(\frac{e^{tS}x-x}{t}\right)_j D_jD_i\phi(x) \\
     & -\frac{1}{2t}\sum_{i,j=1}^{d}\left[\int_{\R^d}K(\psi,t)\psi_j d\psi(e^{tS}x-x)_i+\int_{\R^d}K(\psi,t)\psi_i d\psi(e^{tS}x-x)_j\right]D_jD_i\phi(x) \\
    =& e^{-Bt}A\triangle\phi(x) + \frac{t}{2}e^{-Bt}\sum_{i=1}^{d}\sum_{j=1}^{d}\left(\frac{e^{tS}x-x}{t}\right)_i\left(\frac{e^{tS}x-x}{t}\right)_j D_jD_i\phi(x)
  \end{align*}
  for every $t>0$, where we used \cite[Lemma 4.5(3), (1) and (2)]{Otten2014a} for the first, second and third term, respectively. \\
  This yields a simplified representation for $f_t(x)$ for every $t>0$ given by
  \begin{align}
    \label{equ:RepresentationOfftxTaylorExpansion}
    \begin{split}
      f_t(x) =& e^{-Bt}A\triangle\phi(x) + e^{-Bt}\sum_{i=1}^{d}\left(\frac{e^{tS}x-x}{t}\right)_i D_i\phi(x) + \left(\frac{e^{-Bt}-I_N}{t}\right)\phi(x) \\
              & +\frac{t}{2}e^{-Bt}\sum_{i=1}^{d}\sum_{j=1}^{d}\left(\frac{e^{tS}x-x}{t}\right)_i\left(\frac{e^{tS}x-x}{t}\right)_j D_jD_i\phi(x) \\
              & +\frac{1}{t}\int_{\R^d}K(\psi,t)R_{x,2}(e^{tS}x-x-\psi)d\psi.
    \end{split}
  \end{align}
  (a): Using $\lim_{t\downarrow 0}\frac{e^{tX}-I}{t}=X$ for $X=-B,S$ and $\lim_{t\downarrow 0}e^{-Bt}=I_N$ we obtain
  \begin{align*}
     \left[A_p\phi\right](x)
    =& \lim_{t\downarrow 0}\frac{\left[T(t)\phi\right](x)-\phi(x)}{t}
    =  \lim_{t\downarrow 0}f_t(x) \\
    =& A\triangle\phi(x) + \lim_{t\downarrow 0}\sum_{i=1}^{d}\left(\frac{e^{tS}x-x}{t}\right)_i D_i\phi(x) +\lim_{t\downarrow 0}\left(\frac{e^{-Bt}-I_N}{t}\right)\phi(x)\\
     & +\lim_{t\downarrow 0}\frac{t}{2}\sum_{i=1}^{d}\sum_{j=1}^{d}\left(\frac{e^{tS}x-x}{t}\right)_i\left(\frac{e^{tS}x-x}{t}\right)_j D_jD_i\phi(x) \\
     & +\lim_{t\downarrow 0}\frac{1}{t}\int_{\R^d}K(\psi,t)R_{x,2}(e^{tS}x-x-\psi)d\psi \\
    =& A\triangle\phi(x) + \left\langle Sx,\nabla\phi(x)\right\rangle - B\phi(x) = \left[\L_{\infty}\phi\right](x) = f(x),
  \end{align*}
  i.e. $f_t(x)\rightarrow f(x)$ pointwise for a.e. $x\in\R^d$ as $t\downarrow 0$, provided that the last limit tends to zeros. This can be seen as follows: 
  Using \eqref{equ:EstimateRemainderTaylorExpansion}, the inequality
  \begin{align}
    \label{equ:TriangleInequalityPsiX}
    \left(\left|e^{tS}x-x\right|+\left|\psi\right|\right)^3 \leqslant 4\left(\left|e^{tS}x-x\right|^3+\left|\psi\right|^3\right),\quad x,\psi\in\R^d,\,t>0
  \end{align}
  and the following integral estimate (for $k=0$ and $k=3$)
  \begin{align}
    \label{equ:IntegralEqualityForKthPowerOfPsi}
    \int_{\R^d}\left|K(\psi,t)\right|\left|\psi\right|^k \leqslant \kappa(Y) \aone^{\frac{d}{2}} e^{-\bzero t} \atwo^{\frac{k}{2}}
                                                                   \frac{\Gamma\left(\frac{d+k}{2}\right)}{\Gamma\left(\frac{d}{2}\right)}t^{\frac{k}{2}},\,t>0,\,k\in\N_0
  \end{align}
  with constants $\aone$, $\atwo$ from \eqref{equ:aminamaxazerobzero} and condition number $\kappa(Y):=\left|Y^{-1}\right|\left|Y\right|$ of the transformation 
  matrix $Y\in\C^{N,N}$ from \eqref{cond:A8B}, we obtain
  \begin{align*}
             & \left|\frac{1}{t}\int_{\R^d}K(\psi,t)R_{x,2}(e^{tS}x-x-\psi)d\psi\right| \\
    \leqslant& \frac{1}{t}\int_{\R^d}\left|K(\psi,t)\right|\left|R_{x,2}(e^{tS}x-x-\psi)\right|d\psi \\
    \leqslant& \frac{C_{\beta}C_{\phi}}{t}\int_{\R^d}\left|K(\psi,t)\right| \left|e^{tS}x-x-\psi\right|^3 d\psi \\
    \leqslant& \frac{C_{\beta}C_{\phi}}{t}\int_{\R^d}\left|K(\psi,t)\right| \left(\left|e^{tS}x-x\right|+\left|\psi\right|\right)^3 d\psi \\
    \leqslant& \frac{4 C_{\beta} C_{\phi}}{t}\left[\int_{\R^d}\left|K(\psi,t)\right|d\psi \left|e^{tS}x-x\right|^3+\int_{\R^d}\left|K(\psi,t)\right|_2\left|\psi\right|^3d\psi\right] \\
            =& 4 C_{\beta} C_{\phi} \kappa(Y) \aone^{\frac{d}{2}}e^{-\bzero t}\left[t^2\left|\frac{e^{tS}x-x}{t}\right|^3+\frac{\Gamma\left(\frac{d+3}{2}\right)}{\Gamma\left(\frac{d}{2}\right)}
               \atwo^{\frac{3}{2}}t^{\frac{1}{2}}\right]
  \end{align*} 
  for every $t>0$. Therefore, using $\lim_{t\downarrow 0}\frac{e^{tS}x-x}{t}=Sx$ once more, the right hand side vanishes for a.e. $x\in\R^d$ as $t\downarrow 0$. Note, 
  that estimate \eqref{equ:TriangleInequalityPsiX} follows from a discrete version of H\"older's inequality. The integral estimate \eqref{equ:IntegralEqualityForKthPowerOfPsi} 
  can be proved in the same way as in \cite[Lemma 4.3(1)]{Otten2014a}. \\
  (b): Given some $\varepsilon>0$ we choose $t_0=t_0(\varepsilon)>0$ such that for every $0<t\leqslant t_0$
  \begin{align}
    \label{equ:ChoiceOfEpsilon}
    \left|\frac{e^{tX}-I}{t}\right|\leqslant|X|+\varepsilon,\quad \left|e^{-Bt}\right|,\,e^{-\bzero t}\leqslant 1+\varepsilon\quad\text{and}\quad t(|S|+\varepsilon)\leqslant\frac{1}{2}
  \end{align}
  for both, $X=S$ and $X=-B$. Then \eqref{equ:RepresentationOfftxTaylorExpansion} yields
  \begin{align}
                \left|f_t(x)\right|
    \leqslant&  \left|e^{-tB}\right| |A|\sum_{i=1}^{d}\left|D_i^2\phi(x)\right| +\left|e^{-tB}\right|\sum_{i=1}^{d}\left|\frac{e^{tS}-I_d}{t}\right|\left|x\right|\left|D_i\phi(x)\right| \nonumber\\
             & +\left|\frac{e^{-Bt}-I_N}{t}\right|\left|\phi(x)\right|+\frac{t}{2}\left|e^{-tB}\right|\sum_{i=1}^{d}\sum_{j=1}^{d}\left|\frac{e^{tS}-I_d}{t}\right|^2\left|x\right|^2\left|D_jD_i\phi(x)\right| \nonumber\\
             & +\left|\frac{1}{t}\int_{\R^d}K(\psi,t)R_{x,2}(e^{tS}x-x-\psi)d\psi\right| \nonumber\\
    \leqslant& (1+\varepsilon)|A|\sum_{i=1}^{d}\left|D_i^2\phi(x)\right| 
               +(1+\varepsilon)\sum_{i=1}^{d}\left(|S|+\varepsilon\right)|x|\left|D_i\phi(x)\right| \label{equ:PointwiseEstimateOfft}\\
             & + \left(|B|+\varepsilon\right)\left|\phi(x)\right| +(1+\varepsilon)\frac{t_0}{2}\sum_{i=1}^{d}\sum_{j=1}^{d}\left(|S|+\varepsilon\right)^2\left|x\right|^2\left|D_jD_i\phi(x)\right| \nonumber\\
             & +\left|\frac{1}{t}\int_{\R^d}K(\psi,t)R_{x,2}(e^{tS}x-x-\psi)d\psi\right| \nonumber
  \end{align}
  for every $0<t\leqslant t_0$. Now the first four terms do not depend on $t$ any more. In particular, since $\phi\in\S(\R^d,\C^N)$, the first four terms belong to 
  $L^p(\R^d,\C^N)$. Therefore, it remains to estimate the last term in such a way, that the bound doesn't depend on $t$ any more and belongs to $L^p(\R^d,\C^N)$ as 
  a function of $x$. For this purpose, we must handle the last term very carefully.  
  \begin{align*}
             & \left|\frac{1}{t}\int_{\R^d}K(\psi,t)R_{x,2}(e^{tS}x-x-\psi)d\psi\right| \\
            =& \left|\frac{1}{t}\int_{\R^d}K(\psi,t)\sum_{|\beta|=3}\frac{|\beta|}{\beta!}\left(z-x\right)^{\beta}
               \int_{0}^{1}(1-\tau)^2 D^{\beta}\phi\left(x+\tau\left(z-x\right)\right) d\tau d\psi\right| \\
    \leqslant& \frac{1}{t}\sum_{|\beta|=3}\frac{|\beta|}{\beta!}\int_{\R^d}\left|K(\psi,t)\right|\left|z-x\right|^{|\beta|}
               \int_{0}^{1}(1-\tau)^2 \left|D^{\beta}\phi\left(x+\tau\left(z-x\right)\right)\right| d\tau d\psi \\
    \leqslant& \frac{4 C_{\beta}}{t}\int_{\R^d}\left|K(\psi,t)\right|\left(\left|e^{tS}x-x\right|^3+\left|\psi\right|^3\right)
               \max_{|\beta|=3}\sup_{\tau\in[0,1]}\left|D^{\beta}\phi\left(x+\tau\left(z-x\right)\right)\right|d\psi,
  \end{align*}
  where $z:=e^{tS}x-\psi$, $C_{\beta}=\sum_{|\beta|=3}\frac{1}{\beta!}$. We now must distinguish between four cases: Let $R\geqslant 1$ be arbitrary.\\
  \textbf{Case 1:} ($|x|\geqslant R$, $|\psi|\leqslant \frac{|x|}{4}$). In this case we use $\phi\in\S(\R^d,\C^N)$. 
  Given $\varepsilon>0$ and choose $t_0=t_0(\varepsilon)>0$ as in \eqref{equ:ChoiceOfEpsilon}. From $|x|\geqslant R$, $|\psi|\leqslant \frac{|x|}{4}$ 
  and \eqref{equ:ChoiceOfEpsilon} we obtain for every $\tau\in[0,1]$
  \begin{align*}
             & \left|x+\tau\left(e^{tS}x-x-\psi\right)\right| 
    \geqslant  \left|x\right| - \tau\left|e^{tS}x-x\right| - \tau\left|\psi\right|
    \geqslant  \left|x\right| - \left|e^{tS}x-x\right| - \left|\psi\right| \\
    \geqslant& \left(1-t\left|\frac{e^{tS}-I_d}{t}\right|\right)\left|x\right| - \left|\psi\right|
    \geqslant  \left(1-t\left(|S|+\varepsilon\right)\right)\left|x\right| - \left|\psi\right| 
    \geqslant  \frac{\left|x\right|}{2} - \left|\psi\right| 
    \geqslant  \frac{\left|x\right|}{4}.
  \end{align*}
  Moreover, since $\phi\in\S(\R^d,\C^N)$, we have
  \begin{align*}
    \forall\,\alpha,\beta\in\N_0^d\;\exists\,C_{\alpha,\beta}>0:\;\left|y^{\alpha} D^{\beta}\phi(y)\right|\leqslant C_{\alpha,\beta}\;\forall\,y\in\R^d,
  \end{align*}
  and therefore, for arbitrary $R_0>0$ it holds
  \begin{align}
    \label{equ:PhiRapidlyDecaying}
    \left|D^{\beta}\phi(y)\right|\leqslant C_{\alpha,\beta}\left|y\right|^{-\left|\alpha\right|}\;\forall\,y\in\R^d,\,|y|\geqslant R_0.
  \end{align}
  Thus, using \eqref{equ:IntegralEqualityForKthPowerOfPsi} with $k=0$ and $k=3$, we obtain, $z:=e^{tS}x-\psi$
  \begin{align*}
             & \frac{4 C_{\beta}}{t}\int_{|\psi|\leqslant\frac{|x|}{4}}\left|K(\psi,t)\right|\left(\left|e^{tS}x-x\right|^3+\left|\psi\right|^3\right)
               \max_{|\beta|=3\atop\tau\in[0,1]}\left|D^{\beta}\phi\left(x+\tau\left(z-x\right)\right)\right|d\psi \\
    \leqslant& 4 C_{\beta}\int_{|\psi|\leqslant\frac{|x|}{4}}\left|K(\psi,t)\right|\left(t^2\left|\frac{e^{tS}-I_d}{t}\right|^3|x|^3+\frac{1}{t}\left|\psi\right|^3\right) \\
             & \cdot\max_{|\beta|=3}\sup_{\tau\in[0,1]}C_{\alpha,\beta}\left|x+\tau\left(e^{tS}x-x-\psi\right)\right|^{-\left|\alpha\right|}d\psi \\
    \leqslant& 4 C_{\beta}\int_{|\psi|\leqslant\frac{|x|}{4}}\left|K(\psi,t)\right|\left(t^2\left(|S|+\varepsilon\right)^3|x|^3+\frac{1}{t}\left|\psi\right|^3\right)
               \max_{|\beta|=3}C_{\alpha,\beta}4^{\left|\alpha\right|}\left|x\right|^{-\left|\alpha\right|}d\psi \\
    \leqslant& 4^{|\alpha|+1} C_{\beta}C_{\phi}\bigg[t^2\left(|S|+\varepsilon\right)^3\left|x\right|^{-\left(|\alpha|-3\right)}\int_{\R^d}\left|K(\psi,t)\right|d\psi \\
             & +\frac{1}{t}\left|x\right|^{-\left|\alpha\right|}\int_{\R^d}\left|K(\psi,t)\right|\left|\psi\right|^3 d\psi\bigg] \\
    \leqslant& 4^{|\alpha|+1} C_{\beta}C_{\phi} \kappa(Y) \aone^{\frac{d}{2}}e^{-\bzero t}\left[t^2\left(|S|+\varepsilon\right)^3\left|x\right|^{-\left(|\alpha|-3\right)}
               +t^{\frac{1}{2}}\left|x\right|^{-\left|\alpha\right|}\frac{\Gamma\left(\frac{d+3}{2}\right)}{\Gamma\left(\frac{d}{2}\right)}
               \atwo^{\frac{3}{2}}\right] \\
    \leqslant& 4^{|\alpha|+1} C_{\beta}C_{\phi} \kappa(Y) \aone^{\frac{d}{2}}(1+\varepsilon)\left[t_0^2\left(|S|+\varepsilon\right)^3
               +t_0^{\frac{1}{2}}\frac{1}{R^3}\frac{\Gamma\left(\frac{d+3}{2}\right)}{\Gamma\left(\frac{d}{2}\right)}
               \atwo^{\frac{3}{2}}\right]\left|x\right|^{-\left(|\alpha|-3\right)} =: h_1(x)
  \end{align*}
  for every $0<t\leqslant t_0$ and $|x|\geqslant R$, where $C_{\phi}:=\max_{|\beta|=3}C_{\alpha,\beta}$. Here, we must choose $|\alpha|>\frac{d}{p}+3$ to guarantee the 
  $L^p$--integrability of $h_1(x)$ in $|x|\geqslant R$, since
  \begin{align}
    \label{equ:Aussenraumintegral}
    \int_{a}^{\infty}s^{-n}ds = \frac{a^{1-n}}{n-1},\;n\in\N\text{ with }n>1,\;a\in\R\text{ with }a>0,
  \end{align}
  and
  \begin{align*}
    \int_{|x|\geqslant R}|x|^{-(|\alpha|-3)p}dx = \frac{2\pi^{\frac{d}{2}}}{\Gamma\left(\frac{d}{2}\right)}\int_{R}^{\infty}r^{-\left((|\alpha|-3)p-(d-1)\right)}dr.
  \end{align*}
  \textbf{Case 2:} ($|x|\geqslant R$, $|\psi|\geqslant\frac{|x|}{4}$). In this case we must use that $K(\cdot,t)\in\S(\R^d,\C^N)$. First of all, using $e^{-s^2}\in\S(\R,\R)$, i.e.
  \begin{align*}
    \forall\,m\in\N_0\;\forall\,R>0\;\exists\,C_{R,m}>0:\;\left|e^{-s^2}\right|\leqslant C_{R,m}\left|s\right|^{-m}\;\forall\,|s|\geqslant R,
  \end{align*}
  \eqref{equ:Aussenraumintegral} and the constants $\azero$, $\amin$ and $\amax$ from \eqref{equ:aminamaxazerobzero}, we deduce
  \begin{align*}
             & \int_{|\psi|\geqslant\frac{|x|}{4}}\left|K(\psi,t)\right|\left|\psi\right|^k d\psi
    \leqslant  \kappa(Y)\int_{|\psi|\geqslant\frac{|x|}{4}}\left(4\pi t\amin\right)^{-\frac{d}{2}}e^{-\bzero t-\frac{\azero}{4t\amax^2}|\psi|^2}\left|\psi\right|^k d\psi \\
            =& \kappa(Y)\left(4\pi t\amin\right)^{-\frac{d}{2}}\frac{2\pi^{\frac{d}{2}}}{\Gamma\left(\frac{d}{2}\right)}e^{-\bzero t}
               \int_{\frac{|x|}{4}}^{\infty}r^{d-1}e^{-\frac{\azero}{4t\amax^2}r^2}r^k dr \\
            =& \kappa(Y)\left(\frac{\amax^2}{\amin\azero}\right)^{\frac{d+k}{2}}\left(4t\amin\right)^{\frac{k}{2}}\frac{2}{\Gamma\left(\frac{d}{2}\right)}e^{-\bzero t}
               \int_{\left(\frac{\azero}{4t\amax^2}\right)^{\frac{1}{2}}\frac{|x|}{4}}^{\infty}s^{d-1}e^{-s^2}s^k ds \\
    \leqslant& \kappa(Y)\left(\frac{\amax^2}{\amin\azero}\right)^{\frac{d+k}{2}}\left(4t\amin\right)^{\frac{k}{2}}\frac{2}{\Gamma\left(\frac{d}{2}\right)}e^{-\bzero t}
               \int_{\left(\frac{\azero}{4t\amax^2}\right)^{\frac{1}{2}}\frac{|x|}{4}}^{\infty}s^{d-1+k-m} ds \\
            =& \frac{2 \kappa(Y) \aone^{\frac{d+k}{2}}\left(4t\amin\right)^{\frac{k}{2}} }{(m-d-k)\Gamma\left(\frac{d}{2}\right)}e^{-\bzero t}
               \left[\left(\frac{1}{\atwo t}\right)^{\frac{1}{2}}\frac{|x|}{4}\right]^{-(m-d-k)}
            =: C t^{\frac{m-d}{2}} e^{-\bzero t} \left|x\right|^{-(m-d-k)}
  \end{align*}
  whenever $m\geqslant d+k+1$. Therefore, we obtain for $0<t\leqslant t_0$, $z:=e^{tS}x-\psi$
  \begin{align*}
             & \frac{4 C_{\beta}}{t}\int_{|\psi|\geqslant\frac{|x|}{4}}\left|K(\psi,t)\right|\left(\left|e^{tS}x-x\right|^3+\left|\psi\right|^3\right)
               \max_{|\beta|=3\atop\tau\in[0,1]}\left|D^{\beta}\phi\left(x+\tau\left(z-x\right)\right)\right|d\psi \\
    \leqslant& \frac{4 C_{\beta} C_{\phi}}{t}\int_{|\psi|\geqslant\frac{|x|}{4}}\left|K(\psi,t)\right|\left(t^3\left|\frac{e^{tS}-I_d}{t}\right|^3 4^3+1\right)|\psi|^3 d\psi \\
    \leqslant& \frac{4 C_{\beta} C_{\phi}}{t}\left(4^3 t_0^3(|S|+\varepsilon)^3+1\right)\int_{|\psi|\geqslant\frac{|x|}{4}}\left|K(\psi,t)\right|\left|\psi\right|^3 d\psi \\
    \leqslant& 4 C_{\beta} C_{\phi}\left(4^3 t_0^3(|S|+\varepsilon)^3+1\right)C t^{\frac{m-d-2}{2}} e^{-\bzero t} \left|x\right|^{-(m-d-3)} \\
    \leqslant& 4 C_{\beta} C_{\phi}\left(4^3 t_0^3(|S|+\varepsilon)^3+1\right)C t_0^{\frac{m-d-2}{2}} (1+\varepsilon) \left|x\right|^{-(m-d-3)}=:h_2(x)
  \end{align*}
  for every $0<t\leqslant t_0$ and $|x|\geqslant R$, where $C_{\phi}:=\max_{|\beta|=3}\sup_{y\in\R^d}\left|D^{\beta}\phi(y)\right|$. Here, we must choose $m>\frac{d}{p}+d+3$ 
  to guarantee $L^p$-integrability in $|x|\geqslant R$. \\
  \textbf{Case 3:} ($|x|\leqslant R$, $|\psi|\geqslant\frac{|x|}{4}$). In this case we use that Schwartz functions, as e.g. $\phi$ and their derivatives, are bounded on compact sets, e.g. 
  on $B_R(0)$. Using \eqref{equ:IntegralEqualityForKthPowerOfPsi} with $k=3$, we obtain with $z:=e^{tS}x-\psi$
  \begin{align*}
             & \frac{4 C_{\beta}}{t}\int_{|\psi|\geqslant\frac{|x|}{4}}\left|K(\psi,t)\right|\left(\left|e^{tS}x-x\right|^3+\left|\psi\right|^3\right)
               \max_{|\beta|=3\atop\tau\in[0,1]}\left|D^{\beta}\phi\left(x+\tau\left(z-x\right)\right)\right|d\psi \\
    \leqslant& \frac{4 C_{\beta}C_{\phi}}{t}\int_{|\psi|\geqslant\frac{|x|}{4}}\left|K(\psi,t)\right|\left(t^3\left|\frac{e^{tS}-I_d}{t}\right|^3|x|^3+\left|\psi\right|^3\right) d\psi \\
    \leqslant& \frac{4 C_{\beta}C_{\phi}}{t}\left(4^3 t_0^3(|S|+\varepsilon)^3+1\right)\int_{|\psi|\geqslant\frac{|x|}{4}}\left|K(\psi,t)\right|\left|\psi\right|^3 d\psi \\
    \leqslant& \frac{4 C_{\beta}C_{\phi}}{t}\left(4^3 t_0^3(|S|+\varepsilon)^3+1\right)\int_{\R^d}\left|K(\psi,t)\right|\left|\psi\right|^3 d\psi \\
    \leqslant& 4 C_{\beta}C_{\phi}\left(4^3 t_0^3(|S|+\varepsilon)^3+1\right)\kappa(Y)\aone^{\frac{d}{2}}e^{-\bzero t}\atwo^{\frac{3}{2}}
               \frac{\Gamma\left(\frac{d+3}{2}\right)}{\Gamma\left(\frac{d}{2}\right)}t^{\frac{1}{2}} \\
    \leqslant& 4 C_{\beta}C_{\phi}\left(4^3 t_0^3(|S|+\varepsilon)^3+1\right)\kappa(Y)\aone^{\frac{d}{2}}(1+\varepsilon)\atwo^{\frac{3}{2}}
               \frac{\Gamma\left(\frac{d+3}{2}\right)}{\Gamma\left(\frac{d}{2}\right)}t_0^{\frac{1}{2}} =: h_3
  \end{align*}
  for every $0<t\leqslant t_0$ and $|x|\leqslant R$, where $C_{\phi}:=\max_{|\beta|=3}\sup_{y\in\R^d}\left|D^{\beta}\phi(y)\right|$. \\
  \textbf{Case 4:}  ($|x|\leqslant R$, $|\psi|\leqslant\frac{|x|}{4}$). This case is similar to case 3. Using \eqref{equ:ChoiceOfEpsilon} and \eqref{equ:IntegralEqualityForKthPowerOfPsi} 
  with $k=0$ and $k=3$, we obtain for $z:=e^{tS}x-\psi$
  \begin{align*}
             & \frac{4 C_{\beta}}{t}\int_{|\psi|\leqslant\frac{|x|}{4}}\left|K(\psi,t)\right|\left(\left|e^{tS}x-x\right|^3+\left|\psi\right|^3\right)
               \max_{|\beta|=3\atop\tau\in[0,1]}\left|D^{\beta}\phi\left(x+\tau\left(z-x\right)\right)\right|d\psi \\
    \leqslant& 4 C_{\beta}C_{\phi}\int_{|\psi|\leqslant\frac{|x|}{4}}\left|K(\psi,t)\right|
               \left(t^2\left|\frac{e^{tS}-I_d}{t}\right|^3|x|^3+\frac{1}{t}\left|\psi\right|^3\right) d\psi \\
    \leqslant& 4 C_{\beta}C_{\phi}\left[t_0^2\left(|S|+\varepsilon\right)^3 R^3\int_{\R^d}\left|K(\psi,t)\right|d\psi
               +\frac{1}{t}\int_{\R^d}\left|K(\psi,t)\right|\left|\psi\right|^3 d\psi\right] \\
    \leqslant& 4 C_{\beta}C_{\phi}\kappa(Y)\aone^{\frac{d}{2}}e^{-\bzero t}\left[t_0^2\left(|S|+\varepsilon\right)^3R^3+\frac{\Gamma\left(\frac{d+3}{2}\right)}{\Gamma\left(\frac{d}{2}\right)}
               \atwo^{\frac{3}{2}}t^{\frac{1}{2}}\right] \\
    \leqslant& 4 C_{\beta}C_{\phi}\kappa(Y)\aone^{\frac{d}{2}}(1+\varepsilon)\left[t_0^2\left(|S|+\varepsilon\right)^3R^3+\frac{\Gamma\left(\frac{d+3}{2}\right)}{\Gamma\left(\frac{d}{2}\right)}
               \atwo^{\frac{3}{2}}t_0^{\frac{1}{2}}\right] =: h_4
  \end{align*}
  for every $0<t\leqslant t_0$ and $|x|\leqslant R$, where $C_{\phi}:=\max_{|\beta|=3}\sup_{y\in\R^d}\left|D^{\beta}\phi(y)\right|$. \\
  Now choosing $|\alpha|=\frac{d}{p}+4$ and $m=\frac{d}{p}+d+4$ and defining
  \begin{align*}
    h:\R^d\rightarrow\R,\quad h(x):=\begin{cases}\max\{h_3,h_4\}&\text{, }|x|\leqslant R\\\max\{h_1(x),h_2(x)\}&\text{, }|x|\geqslant R\end{cases}
  \end{align*}
  we deduce from \eqref{equ:PointwiseEstimateOfft}
  \begin{align*}
               \left|f_t(x)\right|
    \leqslant&  (1+\varepsilon)|A|\sum_{i=1}^{d}\left|D_i^2\phi(x)\right| + (1+\varepsilon)\sum_{i=1}^{d}\left(|S|+\varepsilon\right)|x|\left|D_i\phi(x)\right| 
               + (|B|+\varepsilon)\left|\phi(x)\right|\\
             & +(1+\varepsilon)\frac{t_0}{2}\sum_{i=1}^{d}\sum_{j=1}^{d}\left(|S|+\varepsilon\right)^2\left|x\right|^2\left|D_jD_i\phi(x)\right| + h(x) =: g(x)
  \end{align*}
  for every $0<t\leqslant t_0$. \\
  (c): Using the decomposition
  \begin{align*}
    \left\|g\right\|_{L^p(\R^d,\C^N)}^p = \int_{|x|\geqslant R}\left|g(x)\right|^p dx + \int_{|x|\leqslant R}\left|g(x)\right|^p dx
  \end{align*}
  and \eqref{equ:PhiRapidlyDecaying} since $\phi\in\S(\R^d,\C^N)$, we deduce $g\in L^p(\R^d,\R)$ and the application of dominated convergence is justified. \\ 
  \enum{3}: The proof can partially be found in \cite[II.2.13]{EngelNagel2000}. Let $\phi\in\S:=\S(\R^d,\C^N)$. \\
  1. Recall the ($d$-dimensional) diffusion semigroup $\left(G(t,0)\right)_{t\geqslant 0}$
  \begin{align*}
    \left[G(t,0)\phi\right](y) :=& \int_{\R^d}H(e^{-tS}y,\xi,t)\phi(\xi)d\xi \\
                                =& \int_{\R^d}\left(4\pi tA\right)^{-\frac{d}{2}}\exp\left(-Bt-\left(4tA\right)^{-1}\left|y-\xi\right|^2\right)\phi(\xi)d\xi,\,t>0
  \end{align*}
  and recall the kernel $K$
  \begin{align*}
    K(\psi,t) = \left(4\pi t A\right)^{-\frac{d}{2}}\exp\left(-Bt-\left(4 t A\right)^{-1}\left|\psi\right|^2\right),
  \end{align*}
  which satisfies $K(\cdot,t)\in\S$ for every $t>0$, see \cite[VI.5.3 Example]{EngelNagel2000}. Then we have
  \begin{align*}
    \left[G(t,0)\phi\right](x)=\left[K(t)*\phi\right](x)
  \end{align*}
  and hence
  \begin{align}
    \label{equ:RelationOfTToDiffusionSemigroup}
    \left[T(t)\phi\right](x) = \left[G(t,0)\phi\right](e^{tS}x) = \left[K(t)*\phi\right](e^{tS}x).
  \end{align}
  2. First we show that
  \begin{align}
    \label{equ:FourierTransformRotationalArgument}
    \left[\F \phi(e^{tS}\cdot)\right](\xi) = \left[\F \phi(\cdot)\right](e^{tS}\xi)\quad\forall\,\phi\in\S,
  \end{align}
  where $\F \phi$ denotes the Fourier transform of $\phi\in\S$. From the transformation theorem (with transformation $\Phi(x)=e^{tS}x$), 
  \eqref{cond:A5} and the definition of the Fourier transform \cite[VI.5.2 Definition]{EngelNagel2000} we obtain
  \begin{align*}
        \left[\F\phi(e^{tS}\cdot)\right](\xi)
    :=& \int_{\R^d}e^{-i\left\langle x,\xi\right\rangle}\phi(e^{tS}x)dx
     =  \int_{\R^d}e^{-i\left\langle e^{-tS}y,\xi\right\rangle}\phi(y)dy \\
     =& \int_{\R^d}e^{-i\left\langle y,e^{tS}\xi\right\rangle}\phi(y)dy
     =  \left[\F\phi(\cdot)\right](e^{tS}\xi).
  \end{align*}
  3. Next we show that
  \begin{align}
    \label{equ:DecompositionByFourierTransform}
    \left[\F\left[T(t)\phi\right](\cdot)\right](\xi) = \left[\F K(\cdot,t)\right](e^{tS}\xi)\cdot \left[\F\phi\right](e^{tS}\xi).
  \end{align}
  From \eqref{equ:RelationOfTToDiffusionSemigroup} and \eqref{equ:FourierTransformRotationalArgument} we obtain for every $t>0$
  \begin{align*}
      \left[\F\left[T(t)\phi\right](\cdot)\right](\xi)
    =& \left[\F\left[K(t)*\phi\right](e^{tS}\cdot)\right](\xi)
    =  \left[\F\left[K(t)*\phi\right](\cdot)\right](e^{tS}\xi) \\
    =& \left[\left(\F K(t)\right)(\cdot)\cdot \left(\F\phi\right)(\cdot)\right](e^{tS}\xi)
    =  \left[\F K(t)\right](e^{tS}\xi) \cdot \left[\F\phi\right](e^{tS}\xi).
  \end{align*}
  4. Since $\phi\in\S$ it follows that $\left[\F\phi\right](\cdot)\in\S$ and thus $\left[\F\phi\right](e^{tS}\cdot)\in\S$ for every $t\geqslant 0$. Analogously, since $K(\cdot,t)\in\S$ for every $t>0$ 
  it follows that $\left[\F K(t)\right](\cdot)\in\S$ and hence $\left[\F K(t)\right](e^{tS}\cdot)\in\S$ for every $t>0$. Using \eqref{equ:DecompositionByFourierTransform} we deduce that
  $\left[\F\left[T_0(t)\phi\right](\cdot)\right](\cdot)\in\S$ for every $t>0$ (since $\S$ is closed under pointwise multiplication), i.e. $\F\left(T(t)\S\right)\subseteq\S$ for every $t>0$ and 
  hence $T(t)\S\subseteq\F^{-1}(\S)=\S$ for every $t>0$, see \cite[II.7.7 The inversion theorem]{Rudin1991}. The case $t=0$ follows directly from the definition of $T$ in 
  \eqref{equ:OrnsteinUhlenbeckSemigroupLp}, that gives $T(0)\S=\S$.
\end{proof}

%  ------------
% | Remark 3.3 |
%  ------------
\begin{remark}
  Indeed, one can show that also $C_{\mathrm{c}}^{\infty}(\R^d,\C^N)$ is a core for $(A_p,\D(A_p))$, but the arguments are slightly different. Since $C_{\mathrm{c}}^{\infty}(\R^d,\C^N)$ 
  is not invariant under the semigroup $\left(T(t)\right)_{t\geqslant 0}$, we cannot apply \cite[II.1.7 Proposition]{EngelNagel2000}. In this case one must perform 
  a direct proof as in \cite[Proposition 3.2]{Metafune2001}.
\end{remark}

%---------------------------------------------------------------------------------------------------------------------------------------------------
%
%  SECTION 4: (The Ornstein-Uhlenbeck operator in $L^p(\R^d,\C^N)$)
%
%---------------------------------------------------------------------------------------------------------------------------------------------------
\sect{Resolvent estimates for formal Ornstein-Uhlenbeck operators in \texorpdfstring{$L^p(\R^d,\C^N)$}{Lp(Rd,CN)}}
\label{sec:TheOrnsteinUhlenbeckOperatorInLp}
%---------------------------------------------------------------------------------------------------------------------------------------------------

In this section we prove resolvent estimates for the operator
\begin{align*}
  \left[\L_{\infty} v\right](x) = A\triangle v(x) + \left\langle Sx,\nabla v(x)\right\rangle - Bv(x),\,x\in\R^d
\end{align*}
in $L^p(\R^d,\C^N)$ for $1<p<\infty$. Defining the formal Ornstein-Uhlenbeck operator
\begin{align*}
  \left[\L_{0} v\right](x) = A\triangle v(x) + \left\langle Sx,\nabla v(x)\right\rangle,\,x\in\R^d,
\end{align*}
we observe that the operator $\L_{\infty}=\L_0-B$ is a constant coefficient perturbation of $\L_0$. Therefore, we equip the operator $\L_{\infty}$ with the domain
\begin{align*}
  \D^p_{\mathrm{loc}}(\L_{0}):=&\left\{v\in W^{2,p}_{\mathrm{loc}}(\R^d,\C^N)\cap L^p(\R^d,\C^N)\mid A\triangle v+\left\langle S\cdot,\nabla v\right\rangle\in L^p(\R^d,\C^N)\right\} \\
                                   =&\left\{v\in W^{2,p}_{\mathrm{loc}}(\R^d,\C^N)\cap L^p(\R^d,\C^N)\mid \L_0 v\in L^p(\R^d,\C^N)\right\}.
\end{align*}
Note that the domain $\D^p_{\mathrm{loc}}(\L_{0})$ of $\L_{\infty}$ does not depend on the matrix $B$.

The following lemma states that $\L_{\infty}:L^p(\R^d,\C^N)\supseteq\D^p_{\mathrm{loc}}(\L_0)\rightarrow L^p(\R^d,\C^N)$ is a closed operator in $L^p(\R^d,\C^N)$ for every $1<p<\infty$. 
This allows us to define the resolvent of $\L_{\infty}$. A proof for the real-valued case, which is based on a local elliptic $L^p$-regularity result 
from \cite[Theorem 9.11]{GilbargTrudinger2010}, can be found in \cite[Lemma 3.1]{MetafunePallaraWacker2002}. The following lemma extends \cite[Lemma 5.11]{Otten2014} 
to general matrices $B\in\C^{N,N}$.

%  -----------
% | Lemma 4.1 |
%  -----------
\begin{lemma}\label{lem:ClosednessOfL0}
  Let the assumption \eqref{cond:A3} be satisfied for $\K=\C$, then the operator $\L_{\infty}:L^p(\R^d,\C^N)\supseteq\D^p_{\mathrm{loc}}(\L_0)\rightarrow L^p(\R^d,\C^N)$ 
  is closed in $L^p(\R^d,\C^N)$ for $1<p<\infty$.
\end{lemma}

%  -------------------
% | Proof (Lemma 4.1) |
%  -------------------
\begin{proof}
  Let $\left(v_n\right)_{n\in\N}$ be such that $v_n\in\D^p_{\mathrm{loc}}(\L_0)$ converges to $v\in L^p(\R^d,\C^N)$ and $\L_{\infty} v_n$ converges to $u\in L^p(\R^d,\C^N)$ 
  both w.r.t. $\left\|\cdot\right\|_{L^p}$. To show the closedness of $\L_{\infty}$ we must verify that $v\in\D^p_{\mathrm{loc}}(\L_0)$ and $\L_{\infty} v=u$ in $L^p(\R^d,\C^N)$. 

  Let $\Omega\subseteq\R^d$ be an open bounded set. From $\L_{\infty} v_n\rightarrow u$ in $L^p(\R^d,\C^N)$ we infer that $\L_{\infty} v_n|_{\Omega}\rightarrow u|_{\Omega}$ 
  in $L^p(\Omega,\C^N)$ and therefore, $\left(\L_{\infty} v_n|_{\Omega}\right)_{n\in\N}$ is a Cauchy sequence in $L^p(\Omega,\C^N)$. Analogously, we deduce from 
  $v_n\rightarrow v$ in $L^p(\R^d,\C^N)$ that $v_n|_{\Omega}\rightarrow v|_{\Omega}$ in $L^p(\Omega,\C^N)$ and thus $\left(v_n|_{\Omega}\right)_{n\in\N}$ 
  is a Cauchy sequence in $L^p(\Omega,\C^N)$. Since $Sx$ is bounded in $\Omega$ by the boundedness of $\Omega$, \cite[Theorem 9.11]{GilbargTrudinger2010} yields 
  that for every $\Omega'\subset\subset\Omega$ there exists some constant $C=C(\Omega',\Omega,p,A,S,d)>0$ such that
  \begin{align*}
             & \left\|v_n|_{\Omega'}-v_m|_{\Omega'}\right\|_{W^{2,p}(\Omega',\C^N)} \\
    \leqslant& C\left(\left\|v_n|_{\Omega}-v_m|_{\Omega}\right\|_{L^p(\Omega,\C^N)}+\left\|\L_{\infty} v_n|_{\Omega}-\L_{\infty} v_m|_{\Omega}\right\|_{L^p(\Omega,\C^N)}\right)
    \leqslant \varepsilon.
  \end{align*}
  Therefore, $\left(v_n|_{\Omega'}\right)_{n\in\N}$ is a Cauchy sequence in $W^{2,p}(\Omega',\C^N)$ and consequently, there exists some $v^{\Omega'}\in W^{2,p}(\Omega',\C^N)$ 
  such that $v_n|_{\Omega'}\rightarrow v^{\Omega'}$ in $W^{2,p}(\Omega',\C^N)$ and hence in particular in $L^p(\Omega',\C^N)$. Moreover, since $v_n\rightarrow v$ 
  in $L^p(\R^d,\C^N)$ we deduce $v_n|_{\Omega'}\rightarrow v|_{\Omega'}$ in $L^p(\Omega',\C^N)$. Therefore, $v^{\Omega'}=v|_{\Omega'}$ in $L^p(\Omega',\C^N)$ and 
  we further infer that $v_n|_{\Omega'}\rightarrow v|_{\Omega'}$ in $W^{2,p}(\Omega',\C^N)$ and $v|_{\Omega'}\in W^{2,p}(\Omega',\C^N)$. 
  
  Now, by the arbitrariness of $\Omega$ and $\Omega'$ we deduce that $v\in W^{2,p}_{\mathrm{loc}}(\R^d,\C^N)$. Moreover, $v_n|_{\Omega'}\rightarrow v|_{\Omega'}\in W^{2,p}(\Omega',\C^N)$ 
  implies $\L_{\infty} v_n|_{\Omega'}\rightarrow\L_{\infty} v|_{\Omega'}$ in $L^p(\Omega',\C^N)$ and hence $\L_{\infty} v|_{\Omega'}=u|_{\Omega'}$ in $L^p(\Omega',\C^N)$. By arbitrariness 
  of $\Omega$ and $\Omega'$ we deduce $\L_{\infty} v=u\in L^p(\R^d,\C^N)$ and thus $v\in\D^p_{\mathrm{loc}}(\L_0)$.
\end{proof}

Since $(\L_{\infty},\D^p_{\mathrm{loc}}(\L_0))$ is a closed operator on the Banach space $L^p(\R^d,\C^N)$ for every $1<p<\infty$, we have the following notion
\begin{align*}
     \sigma(\L_{\infty}) :=& \left\{\lambda\in\C\mid \lambda I-\L_{\infty}\text{ is not bijective}\right\} &&\text{\begriff{spectrum of $\L_{\infty}$},} \\
       \rho(\L_{\infty}) :=& \C\backslash\sigma(\L_{\infty})                                               &&\text{\begriff{resolvent set of $\L_{\infty}$},} \\
  R(\lambda,\L_{\infty}) :=& \left(\lambda I-\L_{\infty}\right)^{-1}\text{, for }\lambda\in\rho(\L_{\infty})      &&\text{\begriff{resolvent of $\L_{\infty}$}.}
\end{align*}

The estimate from the following Lemma \ref{lem:LemmaForUniquenessInDpmax} is crucial for the $L^p$-resolvent estimates in Theorem \ref{thm:UniquenessInDpmax} below. 
This result is a complex-valued version of \cite[Lemma 2.1]{MetafunePallaraVespri2005} and is taken from \cite[Lemma 5.12]{Otten2014}.

%  -----------
% | Lemma 4.2 |
%  -----------
\begin{lemma}\label{lem:LemmaForUniquenessInDpmax}
  Let the assumption \eqref{cond:A3} be satisfied for $\K=\C$. Moreover, let $\Omega\subset\R^d$ be a bounded domain with a $C^2$-boundary 
  or $\Omega=\R^d$, $1<p<\infty$, $v\in W^{2,p}(\Omega,\C^N)\cap W^{1,p}_0(\Omega,\C^N)$ and $\eta\in C^1_b(\Omega,\R)$ be nonnegative, then
  \begin{align*}
             & -\Re\int_{\Omega}\eta\overline{v}^T|v|^{p-2}A\triangle v \\
    \geqslant& (p-1)\Re\int_{\Omega}\eta|v|^{p-2}\sum_{j=1}^{d}\overline{D_j v}^T A D_j v\one_{\{v\neq 0\}}
               +\Re\int_{\Omega}\overline{v}^T\left|v\right|^{p-2}\sum_{j=1}^{d}D_j\eta A D_jv \\ 
             & +(p-2)\Re\int_{\Omega}\eta|v|^{p-4}\sum_{j=1}^{d}\Bigg[\Re\left(\overline{D_j v}^Tv\right)\overline{v}^T-|v|^2\overline{D_j v}^T \Bigg]A D_j v\one_{\{v\neq 0\}}.
  \end{align*}
\end{lemma}

%  ------------
% | Remark 4.3 |
%  ------------
\begin{remark}
  For the parameter regime $2\leqslant p<\infty$ Lemma \ref{lem:LemmaForUniquenessInDpmax} follows directly from the integration by parts formula and therefore, 
  the estimate is satisfied with equality. In this case, the real parts in front of the integrals can also be dropped and the assumption \eqref{cond:A3} is not used. 
  If $1<p<2$, then Lemma \ref{lem:LemmaForUniquenessInDpmax} is satisfied only with inequality, which is a direct consequence of Fatou's lemma. The positivity of 
  the quadratic term, based on \eqref{cond:A3}, is necessary for the application of Fatou's lemma.
\end{remark}

%  -------------------
% | Proof (Lemma 4.2) |
%  -------------------
\begin{proof}
  We only provide the proof for $\Omega\subset\R^d$ bounded. In case $\Omega=\R^d$ integration by parts yields no boundary terms due to decay at infinity 
  and thus it can be treated in an analogous way but without boundary integrals. Let $\Omega\subset\R^d$ be bounded with $C^2$-boundary $\partial\Omega$. \\
  \textbf{Case 1:} ($2\leqslant p<\infty$). Multiplying $-A\triangle v$ from left by $\eta\overline{v}^T\left|v\right|^{p-2}$, integrating over $\Omega$ and 
  using integration by parts formula we obtain
  \begin{align*}
     & -\int_{\Omega}\eta\overline{v}^T |v|^{p-2}A\triangle v
    =  -\sum_{j=1}^{d}\int_{\Omega}\eta\overline{v}^T |v|^{p-2}AD_j^2 v \\
    =& \sum_{j=1}^{d}\int_{\Omega}(D_j\eta)\overline{v}^T |v|^{p-2} A D_j v 
       +\sum_{j=1}^{d}\int_{\Omega}\eta D_j(\overline{v}^T |v|^{p-2}) A D_j v \\
    =& \sum_{j=1}^{d}\int_{\Omega}(D_j\eta)\overline{v}^T|v|^{p-2} A D_j v 
       +(p-1)\sum_{j=1}^{d}\int_{\Omega}\eta |v|^{p-2} \overline{D_j v}^T A D_j v \one_{\{v\neq 0\}} \\
     & +(p-2)\sum_{j=1}^{d}\int_{\Omega}\eta |v|^{p-4}\left[\Re(\overline{D_j v}^T v)\overline{v}^T-|v|^2\overline{D_j v}^T\right]A D_j v \one_{\{v\neq 0\}} \\
    =& (p-1)\int_{\Omega}\eta|v|^{p-2}\sum_{j=1}^{d}\overline{D_j v}^T A D_j v\one_{\{v\neq 0\}}
        +\int_{\Omega}\overline{v}^T\left|v\right|^{p-2}\sum_{j=1}^{d}D_j\eta A D_jv \\ 
     & +(p-2)\int_{\Omega}\eta|v|^{p-4}\sum_{j=1}^{d}\Bigg[\Re\left(\overline{D_j v}^Tv\right)\overline{v}^T-|v|^2\overline{D_j v}^T \Bigg]A D_j v\one_{\{v\neq 0\}}.
  \end{align*}
  Now applying real parts we deduce the desired estimates with equality. In the computations above we used the following auxiliaries: The relation\linebreak $z+\overline{z}=2\Re z$ 
  yields
  \begin{align}
    \label{equ:DiAbsVUpToP}
    \begin{split}
       & D_j\left(\left|v\right|^p\right) 
      =  D_j\left(\left(|v|^2\right)^{\frac{p}{2}}\right)
      =  \frac{p}{2} \left(|v|^2\right)^{\frac{p}{2}-1} D_j\left(|v|^2\right) 
      =  \frac{p}{2} |v|^{p-2}D_j(\overline{v}^T v) \\
      =& \frac{p}{2} |v|^{p-2} \left[\overline{D_j v}^Tv+\overline{v}^T D_j v\right]
      =  \frac{p}{2} |v|^{p-2} \left[\overline{D_j v}^T v+\overline{\overline{D_j v}^T v}^T\right] \\
      =& p |v|^{p-2} \Re\left(\overline{D_j v}^T v\right)
    \end{split}
  \end{align}
  for every $v\in\C^N$, $p\geqslant 2$ and $j=1,\ldots,d$. This formula remains valid for every $p\geqslant 0$ and $v\neq 0$. Using the formula \eqref{equ:DiAbsVUpToP} 
  we obtain for every $v\neq 0$ and $p\geqslant 2$
  \begin{align*}
     & D_j\left(\overline{v}^T |v|^{p-2}\right)
    =  \overline{D_j v}^T |v|^{p-2} + \overline{v}^T D_j\left(|v|^{p-2}\right) \\
    =& \overline{D_j v}^T |v|^{p-2} + (p-2)\overline{v}^T |v|^{p-4} \Re\left(\overline{D_j v}^T v\right) \\
    =& (p-1)|v|^{p-2}\overline{D_j v}^T +(p-2)|v|^{p-4}\left[\Re\left(\overline{D_j v}^T v\right)\overline{v}^T-|v|^2\overline{D_j v}^T\right].
  \end{align*}
  \textbf{Case 2:}  ($1<p<2$). This case is much more involved and one has to be very careful, since the expression $\left|v\right|^p$ is not differentiable at $v=0$ 
  for $1<p<2$. We prove the assertion in three steps.\\
  1. First we consider $v\in C^2(\overline{\Omega},\C^N)\cap C_{\mathrm{c}}(\Omega,\C^N)$. Multiplying $-A\triangle v$ from left by $\eta\overline{v}^T\left(|v|^2+\varepsilon\right)^{\frac{p}{2}-1}$ 
  for some $\varepsilon>0$, integrating over $\Omega$ and using integration by parts formula we obtain
  \begin{align*}
     & -\int_{\Omega}\eta\overline{v}^T\left(|v|^2+\varepsilon\right)^{\frac{p}{2}-1}A\triangle v
    =  -\sum_{j=1}^{d}\int_{\Omega}\eta\overline{v}^T\left(|v|^2+\varepsilon\right)^{\frac{p}{2}-1}A D_j^2 v \\
    =& \sum_{j=1}^{d}\Bigg[\int_{\Omega}D_j\left(\eta\overline{v}^T\left(|v|^2+\varepsilon\right)^{\frac{p}{2}-1}\right)A D_j v
       -\int_{\partial\Omega}\eta\overline{v}^T\left(|v|^2+\varepsilon\right)^{\frac{p}{2}-1}A D_j v \nu^{j}dS\Bigg] \\
    =& \sum_{j=1}^{d}\Bigg[\int_{\Omega}(D_j\eta)\overline{v}^T\left(|v|^2+\varepsilon\right)^{\frac{p}{2}-1}A D_j v
       +\int_{\Omega}\eta D_j\left(\overline{v}^T\left(|v|^2+\varepsilon\right)^{\frac{p}{2}-1}\right)A D_j v\Bigg] \\
    =& \int_{\Omega}\eta\left(|v|^2+\varepsilon\right)^{\frac{p}{2}-2}\left((p-1)|v|^2+\varepsilon\right)\sum_{j=1}^{d}\overline{D_j v}^T A D_j v \\
     & +\int_{\Omega}\overline{v}^T\left(|v|^2+\varepsilon\right)^{\frac{p}{2}-1}\sum_{j=1}^{d}D_j\eta A D_j v \\
     & +(p-2)\int_{\Omega}\eta\left(|v|^2+\varepsilon\right)^{\frac{p}{2}-2}\sum_{j=1}^{d}\left[\Re(\overline{D_j v}^T v)\overline{v}^T-|v|^2\overline{D_j v}^T\right]A D_j v.
  \end{align*}
  The boundary integral vanishes because from $v\in C_{\mathrm{c}}(\Omega,\C^N)$ follows $\overline{v}(x)=0$ for every $x\in\partial\Omega$. Moreover, we used the relations
  \begin{align*}
     & D_j\left(\left(|v|^2+\varepsilon\right)^{\frac{p}{2}-1}\right) 
    =  \left(\frac{p}{2}-1\right)\left(|v|^2+\varepsilon\right)^{\frac{p}{2}-2} D_j\left(|v|^2+\varepsilon\right) \\
    =& (p-2)\left(|v|^2+\varepsilon\right)^{\frac{p}{2}-2} \Re\left(\overline{D_j v}^T v\right),
  \end{align*}
  cf. \eqref{equ:DiAbsVUpToP} for $p=2$, and
  \begin{align*}
     & D_j\left(\overline{v}^T\left(|v|^2+\varepsilon\right)^{\frac{p}{2}-1}\right)
    =  \overline{D_j v}^T\left(|v|^2+\varepsilon\right)^{\frac{p}{2}-1}+\overline{v}^T D_j\left(\left(|v|^2+\varepsilon\right)^{\frac{p}{2}-1}\right) \\
    =& \overline{D_j v}^T \left(|v|^2+\varepsilon\right)^{\frac{p}{2}-1}+(p-2)\left(|v|^2+\varepsilon\right)^{\frac{p}{2}-2}\overline{v}^T\Re\left(\overline{D_j v}^T v\right) \\
    =& \left(|v|^2+\varepsilon\right)^{\frac{p}{2}-2}\left[\overline{D_j v}^T\left(|v|^2+\varepsilon\right)+(p-2)\overline{v}^T\Re\left(\overline{D_j v}^T v\right)\right] \\
    =& \left(|v|^2+\varepsilon\right)^{\frac{p}{2}-2}\left((p-1)|v|^2+\varepsilon\right)\overline{D_j v}^T \\
     & +\left(|v|^2+\varepsilon\right)^{\frac{p}{2}-2}(p-2)\left[\Re\left(\overline{D_j v}^T v\right)\overline{v}^T-|v|^2\overline{D_j v}^T\right].
  \end{align*}
  Note that both formulas are valid for $1<p<2$, $v\in\C$ and $j=1,\ldots,d$ if $\varepsilon>0$ and for $1<p<2$, $v\neq 0$ and $j=1,\ldots,d$ if $\varepsilon=0$. \\
  2. We now apply Lebesgue's dominated convergence theorem, see \cite[A1.21]{Alt2006}: Putting the last two terms of the equation from step 1 to the left hand side, 
  taking the limit $\varepsilon\to 0$ and applying dominated convergence twice we obtain
  \begin{align*}
     & (p-1)\int_{\Omega}\eta |v|^{p-2}\sum_{j=1}^{d}\overline{D_j v}^T A D_j v \one_{\{v\neq 0\}} \\
    =& \lim_{\varepsilon\to 0}\int_{\Omega}\eta\left(|v|^2+\varepsilon\right)^{\frac{p}{2}-2}\left((p-1)|v|^2+\varepsilon\right)\sum_{j=1}^{d}\overline{D_j v}^T A D_j v \\
    =& -\lim_{\varepsilon\to 0}\Bigg[\int_{\Omega}\eta\overline{v}^T\left(|v|^2+\varepsilon\right)^{\frac{p}{2}-1}A\triangle v
       +\int_{\Omega}\overline{v}^T\left(|v|^2+\varepsilon\right)^{\frac{p}{2}-1}\sum_{j=1}^{d}D_j \eta A D_j v \\
     & +(p-2)\int_{\Omega}\eta\left(|v|^2+\varepsilon\right)^{\frac{p}{2}-2}\sum_{j=1}^{d}\left[\Re\left(\overline{D_j v}^T v\right)\overline{v}^T-|v|^2\overline{D_j v}^T\right]A D_j v\Bigg] \\
    =& -\int_{\Omega}\eta\overline{v}^T|v|^{p-2}A\triangle v - \int_{\Omega}\overline{v}^T|v|^{p-2}\sum_{j=1}^{d}D_j \eta A D_j v \\
     & -(p-2)\int_{\Omega}\eta|v|^{p-4}\sum_{j=1}^{d}\left[\Re\left(\overline{D_j v}^T v\right)\overline{v}^T-|v|^2\overline{D_j v}^T\right] A D_j v \one_{\{v\neq 0\}}.
  \end{align*}
  To justify the applications of Lebesgue's theorem, we discuss the assumptions in both cases: First, we define
  \begin{align*}
       f_{\varepsilon} :=& \eta\left(|v|^2+\varepsilon\right)^{\frac{p}{2}-2}\left((p-1)|v|^2+\varepsilon\right)\sum_{j=1}^{d}\overline{D_j v}^T A D_j v \\
       f :=& (p-1)\eta |v|^{p-2}\sum_{j=1}^{d}\overline{D_j v}^T A D_j v.
  \end{align*}
  Using $v\in C^2(\overline{\Omega},\C^N)\cap C_{\mathrm{c}}(\Omega,\C^N)$, $\eta\in C_{\mathrm{b}}(\Omega,\R)$ and $\left(|v|^2+\varepsilon\right)^{\frac{p}{2}-k}\leqslant |v|^{p-2k}$ 
  for $k=1,2$ and $1<p<2$ we obtain that $f_{\varepsilon}$ is dominated by $g$ as follows
  \begin{align*}
               \left|f_{\varepsilon}\right|
            =& \left|\eta\left((p-2)|v|^2\left(|v|^2+\varepsilon\right)^{\frac{p}{2}-2}+\left(|v|^2+\varepsilon\right)^{\frac{p}{2}-1}\right)\sum_{j=1}^{d}\overline{D_j v}^T A D_j v\right| \\
    \leqslant& |\eta|\left(|p-2|+1\right)|v|^{p-2}|A|\sum_{j=1}^{d}\left|D_j v\right|^2 \\
            =& |p-3||A||\eta||v|^{p-2}|\nabla v|^2 \one_{\{v\neq 0\}} \\
    \leqslant& |p-3||A|\left\|\eta\right\|_{\infty}\left\|v\right\|_{\infty}^{p-2}\left\|\nabla v\right\|_{\infty}^2 \one_{\{v\neq 0\}} =: g.
  \end{align*}
  Since $v$ is compactly supported, i.e. $\one_{\{v\neq 0\}}$ is compact, $g$ belongs to $L^1(\Omega,\R)$. In particular, $f_{\varepsilon}\rightarrow f$ pointwise a.e. as $\varepsilon\to 0$. 
  Thus, by dominated convergence, $f_{\varepsilon},f\in L^1(\Omega,\C^N)$ and $f_{\varepsilon}\rightarrow f$ in $L^1(\Omega,\C^N)$ as $\varepsilon\to 0$. Next, consider
  \begin{align*}
    f_{\varepsilon} :=& \overline{v}^T\left(|v|^2+\varepsilon\right)^{\frac{p}{2}-1}\left(\eta A\triangle v+\sum_{j=1}^{d}D_j \eta A D_j v\right) \\
                      & +(p-2)\eta\left(|v|^2+\varepsilon\right)^{\frac{p}{2}-2}\sum_{j=1}^{d}\left[\Re\left(\overline{D_j v}^Tv\right)\overline{v}^T-|v|^2\overline{D_j v}^T\right]A D_j v.\\
    f :=& \overline{v}^T|v|^{p-2}\left(\eta A\triangle v+\sum_{j=1}^{d}D_j \eta A D_j v\right) \\
        & +(p-2)\eta|v|^{p-4}\sum_{j=1}^{d}\left[\Re\left(\overline{D_j v}^T v\right)\overline{v}^T-|v|^2\overline{D_j v}^T\right]A D_j v
  \end{align*}
  Using $v\in C^2(\overline{\Omega},\C^N)\cap C_{\mathrm{c}}(\Omega,\C^N)$, $\eta\in C^1_b(\Omega,\R)$ and $\left(|v|^2+\varepsilon\right)^{\frac{p}{2}-k}\leqslant |v|^{p-2k}$ for $k=1,2$ 
  and $1<p<2$ we obtain that $f_{\varepsilon}$ is dominated by $g$ as follows
  \begin{align*}
               \left|f_{\varepsilon}\right|
    \leqslant& |v|\left(|v|^2+\varepsilon\right)^{\frac{p}{2}-1}\left(|\eta||A||\triangle v|+|A|\sum_{j=1}^{d}|D_j \eta||D_j v|\right) \\
             & +|p-2||\eta|\left(|v|^2+\varepsilon\right)^{\frac{p}{2}-2}\sum_{j=1}^{d}\left[\left|\Re\left(\overline{D_j v}^T v\right)\right||v|+|v|^2|D_j v|\right]|A||D_j v| \\
    \leqslant& |v|^{p-1}\left(|\eta||A||\triangle v|+|A|\sum_{j=1}^{d}|D_j \eta||D_j v|\right)\one_{\{v\neq 0\}} \\
             & +2|p-2||\eta||v|^{p-2}\sum_{j=1}^{d}|D_j v|^2|A|\one_{\{v\neq 0\}} \\
    \leqslant& \bigg[|A|\left\|\eta\right\|_{\infty}\left\|v\right\|_{\infty}^{p-1}\left\|\triangle v\right\|_{\infty}
               +d|A|\left\|\eta\right\|_{1,\infty}\left\|v\right\|_{1,\infty} \\
             & 2d|p-2||A|\left\|\eta\right\|_{\infty}\left\|v\right\|_{\infty}^{p-2}\left\|v\right\|_{1,\infty}^2\bigg]\one_{\{v\neq 0\}} =: g.
  \end{align*}
  Since $v$ is compactly supported, we deduce once more that $g$ belongs to $L^1(\Omega,\R)$. In particular, $f_{\varepsilon}\rightarrow f$ pointwise a.e. as $\varepsilon\to 0$. Thus, 
  by dominated convergence, $f_{\varepsilon},f\in L^1(\Omega,\C^N)$ and $f_{\varepsilon}\rightarrow f$ in $L^1(\Omega,\C^N)$ as $\varepsilon\to 0$. \\
  3. Now let $v\in W^{2,p}(\Omega,\C^N)\cap W^{1,p}_0(\Omega,\C^N)$. In this case we use a density argument and Fatou's lemma, that yields the inequality. Note 
  that we have to take real parts on both sides in order to apply Fatou's lemma. Since $C^2(\overline{\Omega},\C^N)\cap C_{\mathrm{c}}(\Omega,\C^N)$ is a 
  dense subspace of $W^{2,p}(\Omega,\C^N)\cap W^{1,p}_0(\Omega,\C^N)$ w.r.t. $\left\|\cdot\right\|_{W^{2,p}}$, there exists a sequence 
  $v_n\in C^2(\overline{\Omega},\C^N)\cap C_{\mathrm{c}}(\Omega,\C^N)$ such that $v_n\rightarrow v$ w.r.t. $\left\|\cdot\right\|_{W^{2,p}}$ 
  as $n\to\infty$, $n\in\N$. Furthermore, there exists a subset $\N'\subset\N$ such that $v_n\rightarrow v$ and $\nabla v_n\rightarrow \nabla v$ pointwise a.e. 
  as $n\to\infty$, $n\in\N'$. In the following we consider this subsequence $(v_n)_{n\in\N'}\subset C^2(\overline{\Omega},\C^N)\cap C_{\mathrm{c}}(\Omega,\C^N)$: 
  Inserting $v_n$ into the equation from step 2, taking real parts and the limit inferior $n\to\infty$ ($n\in\N'$) on both sides and applying Fatou's lemma 
  on the left hand side we obtain
  \begin{align*}
     & (p-1)\Re\int_{\Omega}\eta|v|^{p-2}\sum_{j=1}^{d}\overline{D_j v}^T A D_j v \one_{\{v\neq 0\}} \\
    =& \int_{\Omega}\lim_{n\to\infty}(p-1)\eta|v_n|^{p-2}\Re\sum_{j=1}^{d}\overline{D_j v_n}^T A D_j v_n \one_{\{v_n\neq 0\}} \\
    =& \int_{\Omega}\underset{n\to\infty}{\liminf}\,(p-1)\eta|v_n|^{p-2}\Re\sum_{j=1}^{d}\overline{D_j v_n}^T A D_j v_n \one_{\{v_n\neq 0\}} \\
    \leqslant& \underset{n\to\infty}{\liminf}\,\int_{\Omega}(p-1)\eta|v_n|^{p-2}\Re\sum_{j=1}^{d}\overline{D_j v_n}^T A D_j v_n \one_{\{v_n\neq 0\}} \\
    =& \underset{n\to\infty}{\liminf}\,\Bigg[-\Re\int_{\Omega}\eta\overline{v_n}^T|v_n|^{p-2}A\triangle v_n-\Re\int_{\Omega}\overline{v_n}^T|v_n|^{p-2}\sum_{j=1}^{d}D_j \eta A D_j v_n \\
     & -(p-2)\Re\int_{\Omega}\eta|v_n|^{p-4}\sum_{j=1}^{d}\left[\Re\left(\overline{D_j v_n}^T v_n\right)\overline{v_n}^T-|v_n|^2\overline{D_j v_n}^T\right]A D_j v_n\one_{\{v_n\neq 0\}}\Bigg] \\
    =& \lim_{n\to\infty}\Bigg[-\Re\int_{\Omega}\eta\overline{v_n}^T|v_n|^{p-2}A\triangle v_n-\Re\int_{\Omega}\overline{v_n}^T|v_n|^{p-2}\sum_{j=1}^{d}D_j \eta A D_j v_n \\
     & -(p-2)\Re\int_{\Omega}\eta|v_n|^{p-4}\sum_{j=1}^{d}\left[\Re\left(\overline{D_j v_n}^T v_n\right)\overline{v_n}^T-|v_n|^2\overline{D_j v_n}^T\right]A D_j v_n\one_{\{v_n\neq 0\}}\Bigg] \\
    =& -\Re\int_{\Omega}\eta\overline{v}^T|v|^{p-2}A\triangle v - \Re\int_{\Omega}\overline{v}^T|v|^{p-2}\sum_{j=1}^{d}D_j \eta A D_j v \\
     & -(p-2)\Re\int_{\Omega}\eta|v|^{p-4}\sum_{j=1}^{d}\left[\Re\left(\overline{D_j v}^T v\right)\overline{v}^T-|v|^2\overline{D_j v}^T\right]A D_j v\one_{\{v\neq 0\}}.
  \end{align*}
  In the first equality we used the fact that $v_n\rightarrow v$ and $\nabla v_n\rightarrow \nabla v$ pointwise a.e. as $n\to\infty$, $n\in\N'$. The last equality 
  can be accepted as follows: Let $f_n\rightarrow f$ in $L^{q}$ and $g_n\rightarrow g$ in $L^p$ with $\frac{1}{p}+\frac{1}{q}=1$, i.e. $q=\frac{p}{p-1}$, 
  then $\int f_n g_n\rightarrow\int fg$ by H\"older's inequality, since
  \begin{align*}
    \int(f_ng_n-fg) &= \int(f_n-f)g+\int f(g_n-g) \\
                    &\leqslant \left\|f_n-f\right\|_{L^q}\left\|q_n\right\|_{L^p}+\left\|f\right\|_{L^q}\left\|g_n-g\right\|_{L^p}\rightarrow 0.
  \end{align*}
  Thus,
  \begin{align*}
    &\overline{v_n}^T|v_n|^{p-2}\overset{L^q}{\rightarrow}\overline{v}^T|v|^{p-2}, &&A\triangle v_n\overset{L^p}{\rightarrow}A\triangle v, \\
    &\overline{v_n}^T|v_n|^{p-2}\overset{L^q}{\rightarrow}\overline{v}^T|v|^{p-2}, &&AD_j v_n\overset{L^p}{\rightarrow}AD_j v, \\
    &|v_n|^{p-4}\Re\left(\overline{D_j v_n}^T v_n\right)\overline{v_n}^T\overset{L^q}{\rightarrow}|v|^{p-4}\Re\left(\overline{D_j v}^T v\right)\overline{v}^T,
     &&A D_j v_n\overset{L^p}{\rightarrow}A D_j v, \\
    &|v_n|^{p-2}\overline{D_j v_n}^T\overset{L^q}{\rightarrow}|v|^{p-2}\overline{D_j v}^T, &&A D_j v_n\overset{L^p}{\rightarrow}A D_j v,
  \end{align*}
  together with $\eta\in C^1_{\mathrm{b}}(\R^d,\R)$ yields the last equality in the above equation. It remains to justify the application of Fatou's lemma, \cite[A1.20]{Alt2006}: Consider
  \begin{align*}
    f_n:=(p-1)\eta|v_n|^{p-2}\Re\sum_{j=1}^{d}\overline{D_j v_n}^T A D_j v_n \one_{\{v_n\neq 0\}},\,n\in\N'.
  \end{align*}
  By H\"older's inequality we have already seen that $\liminf_{n\to\infty}f_n<\infty$ is satisfied. Moreover, $f_n\geqslant 0$ pointwise a.e., since $A$ satisfies assumption 
  \eqref{cond:A3} and $\eta$ is nonnegative. Finally, $f_n\in L^1(\Omega,\R)$, since $v_n\in C^2(\overline{\Omega},\C^N)\cap C_{\mathrm{c}}(\Omega,\C^N)$ and $\eta\in C^1_{\mathrm{b}}(\R^d,\R)$. 
  Thus, by Fatou's, $\liminf_{n\to\infty}f_n\in L^1(\Omega,\R)$ and 
  \begin{align*}
    \int_{\Omega}\underset{n\to\infty}{\liminf}f_n\leqslant \underset{n\to\infty}{\liminf}\int_{\Omega}f_n,
  \end{align*}
  that proves the lemma. Note, that it is in general not possible to apply Lebesgue's theorem in case of $v\in W^{2,p}(\Omega,\C^N)\cap W^{1,p}_0(\Omega,\C^N)$, since one cannot 
  determine a $n$-independent bound for $|f_n|\leqslant g$ a.e. for every $n\in\N'$. In fact, we only know positivity of $f_n$ due to \eqref{cond:A3}, that justifies the application 
  of Fatou's lemma and generates an inequality for $1<p<2$.
\end{proof}

We now prove sharp resolvent estimates for the formal operator $\L_{\infty}$ in $L^p(\R^d,\C^N)$ for $1<p<\infty$, which then yield uniqueness for solutions of the resolvent 
equation for $\L_{\infty}$ in $\D^p_{\mathrm{loc}}(\L_0)$. The techniques are related to \cite[Theorem 2.2, Remark 2.3]{MetafunePallaraVespri2005} for the scalar 
real-valued case and from \cite[Theorem 3.1]{BeynLorenz2008} for $d=2$. In our situation, the proof requires the additional $L^p$-dissipativity condition \eqref{cond:A4DC}. 
The condition seems to be optimal in order to derive resolvent estimates for $\L_{\infty}$ in $L^p(\R^d,\C^N)$ for $1<p<\infty$ and contains an additional restriction on 
the spectrum of the diffusion matrix $A$. 

%  -------------
% | Theorem 4.4 |
%  -------------
\begin{theorem}[Resolvent Estimates for $\L_{\infty}$ in $L^p(\R^d,\C^N)$ with $1<p<\infty$]\label{thm:UniquenessInDpmax}
  Let the assumptions \eqref{cond:A4DC} and \eqref{cond:A5} be satisfied for $1<p<\infty$ and $\K=\C$. Moreover, let $\lambda\in\C$ with 
  $\Re\lambda>\beta_B$, where $\beta_B\in\R$ is from \eqref{equ:betaB}, and let $v_{\star}\in\D^p_{\mathrm{loc}}(\L_0)$ denote a solution of
  \begin{align*}
    \left(\lambda I-\L_{\infty}\right)v=g
  \end{align*}
  in $L^p(\R^d,\C^N)$ for some $g\in L^p(\R^d,\C^N)$. Then $v_{\star}$ is the unique solution in $\D^p_{\mathrm{loc}}(\L_0)$ and satisfies the resolvent estimate
  \begin{align*}
    \left\|v_{\star}\right\|_{L^p(\R^d,\C^N)}\leqslant\frac{1}{\Re\lambda-\beta_B}\left\|g\right\|_{L^p(\R^d,\C^N)}.
  \end{align*}
  In addition, for $1<p\leqslant 2$ the following gradient estimate holds
  \begin{align*}
    \left|v_{\star}\right|_{W^{1,p}(\R^d,\C^N)}\leqslant \frac{d^{\frac{1}{p}}\gamma_A^{-\frac{1}{2}}}{\left(\Re\lambda-\beta_B\right)^{\frac{1}{2}}}\left\|g\right\|_{L^p(\R^d,\C^N)}.
  \end{align*} 
\end{theorem}

%  ------------
% | Remark 4.5 |
%  ------------
\begin{remark}
  \enum{1} Note that the proof deals with cut-off functions. These are necessary because $v\in W^{2,p}_{\mathrm{loc}}(\R^d,\C^N)$ implies 
           that $\nabla v$ and $\triangle v$ are only $p$-integrable over bounded sets in $\R^d$.\\
  \enum{2} The gradient estimate is proved only for $1<p\leqslant 2$ but not for $p>2$. The requirement $1<p\leqslant 2$ appears in a special 
           application of H\"older's inequality. \\
  \enum{3} An $L^p$-dissipativity condition for the operator $\nabla^T\left(Q\nabla v\right)+\left\langle b,\nabla v\right\rangle + av$ in 
           $L^p(\Omega,\C)$ with $1<p<\infty$ can be found in \cite{CialdeaMazya2005} for the scalar case but with complex $b$, namely for constant 
           coefficients $Q\in\C^{d,d}$, $b\in\C^d$, $a\in\C$ with $\Omega\subseteq\R^d$ open in \cite[Theorem 2]{CialdeaMazya2005}, and for 
           variable coefficients $Q_{ij},\,b_j\in C^1(\overline{\Omega},\C)$, $a\in C^0(\overline{\Omega},\C)$ with $\Omega\subset\R^d$ bounded 
           in \cite[Lemma 2]{CialdeaMazya2005}. \\
  \enum{4} The $L^p$-dissipativity condition \eqref{cond:A4DC} needed in Theorem \ref{thm:UniquenessInDpmax} is not easy to interpret and to give it 
           a geometric meaning. For a complete characterization of the $L^p$-dissipativity condition \eqref{cond:A4DC} in terms of the antieigenvalues 
           of the diffusion matrix $A$ we refer to \cite[Theorem 5.18]{Otten2014}. There, it is proved that for fixed $1<p<\infty$ the $L^p$-dissipativity 
           condition \eqref{cond:A4DC} is equivalent to a lower p-dependent bound for the first antieigenvalue of the diffusion matrix $A$.
\end{remark}

%  ---------------------
% | Proof (Theorem 4.4) |
%  ---------------------
\begin{proof}
  Assume $v_{\star}\in\D^p_{\mathrm{loc}}(\L_0)$ satisfies
  \begin{align}
    \label{equ:ResolventEquationVStarG}
    \left(\lambda I-\L_{\infty}\right)v_{\star} = g
  \end{align}
  in $L^p(\R^d,\C^N)$ for some $g\in L^p(\R^d,\C^N)$ with $1<p<\infty$. Let us define
  \begin{align*}
    \eta_n(x)=\eta\left(\frac{x}{n}\right),\quad \eta\in C_{\mathrm{c}}^{\infty}(\R^d,\R),\quad \eta(x)=\begin{cases}1                          &,\,|x|\leqslant 1 \\
                                                                                                          \in[0,1],\,\textrm{smooth} &,\,1<|x|<2 \\
                                                                                                          0                          &,\,|x|\geqslant 2
                                                                                             \end{cases}.
  \end{align*}
  1. Multiplying \eqref{equ:ResolventEquationVStarG} from left by $\eta_n^2 \overline{v_{\star}}^T\left|v_{\star}\right|^{p-2}$ with $1<p<\infty$, 
  integrating over $\R^d$ and taking real parts yields
  \begin{align*}
       \Re\int_{\R^d}\eta_n^2 \left|v_{\star}\right|^{p-2}\overline{v_{\star}}^T g
    =& (\Re\lambda)\int_{\R^d}\eta_n^2\left|v_{\star}\right|^p
       - \Re\int_{\R^d}\eta_n^2\overline{v_{\star}}^T\left|v_{\star}\right|^{p-2} A\triangle v_{\star} \\
     & - \Re\int_{\R^d}\eta_n^2\overline{v_{\star}}^T\left|v_{\star}\right|^{p-2} \sum_{j=1}^{d}(Sx)_j D_j v_{\star} \\
     & + \Re\int_{\R^d}\eta_n^2\overline{v_{\star}}^T\left|v_{\star}\right|^{p-2} B v_{\star}.
  \end{align*}
  2. Using \eqref{cond:A5}, i.e. $-S=S^T$, then integration by parts formula and \eqref{equ:DiAbsVUpToP} imply
  \begin{align*}
    0 =& \frac{1}{p}\int_{\R^d}\eta_n^2\left(\sum_{j=1}^{d}S_{jj}\right)\left|v_{\star}\right|^p
      =  \frac{1}{p}\int_{\R^d}\eta_n^2 \mathrm{div}\left(Sx\right)\left|v_{\star}\right|^p \\
      =& \frac{1}{p}\int_{\R^d}\eta_n^2 \left(\sum_{j=1}^{d}D_j\left((Sx)_j\right)\right)\left|v_{\star}\right|^p
      =  \frac{1}{p}\sum_{j=1}^{d}\int_{\R^d}\eta_n^2 D_j\left((Sx)_j\right)\left|v_{\star}\right|^p \\
      =& -\frac{1}{p}\sum_{j=1}^{d}\int_{\R^d}D_j\left(\eta_n^2\right) (Sx)_j \left|v_{\star}\right|^p 
         -\frac{1}{p}\sum_{j=1}^{d}\int_{\R^d}\eta_n^2 (Sx)_j D_j\left(\left|v_{\star}\right|^p\right) \\
      =& -\frac{2}{p}\sum_{j=1}^{d}\int_{\R^d}\eta_n (D_j \eta_n) (Sx)_j \left|v_{\star}\right|^p
         -\sum_{j=1}^{d}\int_{\R^d}\eta_n^2 (Sx)_j \Re\left(\overline{D_j v_{\star}}^T v_{\star}\right)\left|v_{\star}\right|^{p-2} \\
      =& -\frac{2}{p}\int_{\R^d}\eta_n \left|v_{\star}\right|^p \sum_{j=1}^{d}(D_j \eta_n)(Sx)_j 
         -\Re\int_{\R^d}\eta_n^2 \overline{v_{\star}}^T\left|v_{\star}\right|^{p-2}\sum_{j=1}^{d}(Sx)_j D_j v_{\star}.
  \end{align*}
  Since \eqref{cond:A4DC} implies \eqref{cond:A3} an application of Lemma \ref{lem:LemmaForUniquenessInDpmax} (with $\Omega=\R^d$, $\eta=\eta_n^2$) yields
  \begin{align*}
             &   \Re\int_{\R^d}\eta_n^2 \left|v_{\star}\right|^{p-2} \overline{v_{\star}}^T g \\
    \geqslant&   (\Re\lambda)\int_{\R^d}\eta_n^2\left|v_{\star}\right|^p
               + \Re\int_{\R^d}2\eta_n \overline{v_{\star}}^T\left|v_{\star}\right|^{p-2}\sum_{j=1}^{d}D_j \eta_n A D_j v_{\star} \\
             & + (p-1)\Re\int_{\R^d}\eta_n^2\left|v_{\star}\right|^{p-2}\sum_{j=1}^{d}\overline{D_j v_{\star}}^T A D_j v_{\star} 
               + \frac{2}{p} \int_{\R^d}\eta_n\left|v_{\star}\right|^p \sum_{j=1}^{d}(D_j \eta_n)(Sx)_j \\
             & + (p-2)\Re\int_{\R^d}\eta_n^2\left|v_{\star}\right|^{p-4}\sum_{j=1}^{d}\left[\Re\left(\overline{D_j v_{\star}}^T v_{\star}\right)\overline{v_{\star}}^T-|v_{\star}|^2\overline{D_j v_{\star}}^T\right]A D_j v_{\star} \\
             & + \Re\int_{\R^d}\eta_n^2\overline{v_{\star}}^T\left|v_{\star}\right|^{p-2} B v_{\star}.
  \end{align*}
  3. Putting the $2$nd and $4$th term from the right hand to the left hand side yields
  \begin{align*}
             & (\Re\lambda)\int_{\R^d}\eta_n^2\left|v_{\star}\right|^p 
               + (p-1)\Re\int_{\R^d}\eta_n^2\left|v_{\star}\right|^{p-2}\sum_{j=1}^{d}\overline{D_j v_{\star}}^T A D_j v_{\star} \\
             & + (p-2)\Re\int_{\R^d}\eta_n^2\left|v_{\star}\right|^{p-4}\sum_{j=1}^{d}\left[\Re\left(\overline{D_j v_{\star}}^T v_{\star}\right)\overline{v_{\star}}^T-|v_{\star}|^2\overline{D_j v_{\star}}^T\right]A D_j v_{\star} \\
             & + \Re\int_{\R^d}\eta_n^2\overline{v_{\star}}^T\left|v_{\star}\right|^{p-2} B v_{\star} \\
    \leqslant& \Re\int_{\R^d}\eta_n^2 \left|v_{\star}\right|^{p-2} \overline{v_{\star}}^T g 
               - \Re\int_{\R^d}2\eta_n \overline{v_{\star}}^T\left|v_{\star}\right|^{p-2}\sum_{j=1}^{d}D_j \eta_n A D_j v_{\star} \\
             & - \frac{2}{p} \int_{\R^d}\eta_n\left|v_{\star}\right|^p \sum_{j=1}^{d}(D_j \eta_n)(Sx)_j.
  \end{align*}
  For the $1$st term on the right hand side we use $\Re z\leqslant |z|$ and H{\"o}lder's inequality (with $q$ such that $\frac{1}{p}+\frac{1}{q}=1$)
  \begin{align*}
             & \Re\int_{\R^d}\eta_n^2 \left|v_{\star}\right|^{p-2}\overline{v_{\star}}^T g
            =  \int_{\R^d}\eta_n^2 \left|v_{\star}\right|^{p-2}\Re\left(\overline{v_{\star}}^T g\right) \\
    \leqslant& \int_{\R^d}\eta_n^2 \left|v_{\star}\right|^{p-1}\left|g\right|
    \leqslant  \left(\int_{\R^d}\left(\eta_n^{\frac{2(p-1)}{p}}\left|v_{\star}\right|^{p-1}\right)^{\frac{p}{p-1}}\right)^{\frac{p-1}{p}} \left(\int_{\R^d}\left(\eta_n^{\frac{2}{p}}\left|g\right|\right)^p\right)^{\frac{1}{p}} \\
            =& \left(\int_{\R^d}\eta_n^{2}\left|v_{\star}\right|^p\right)^{\frac{p-1}{p}} \left(\int_{\R^d}\eta_n^2\left|g\right|^p\right)^{\frac{1}{p}}
  \end{align*}
  For the $2$nd term we use $\Re z\leqslant |z|$, H{\"o}lder's inequality (with $p=q=2$) and Cauchy's inequality (with $\varepsilon>0$)
  \begin{align*}
             & - \Re\int_{\R^d}2\eta_n\overline{v_{\star}}^T\left|v_{\star}\right|^{p-2}\sum_{j=1}^{d}D_j \eta_n A D_jv_{\star} \\
    \leqslant& 2|A|\int_{\R^d}\eta_n\left|v_{\star}\right|^{p-1}\sum_{j=1}^{d}\left|D_j \eta_n\right| \left|D_j v_{\star}\right|
    \leqslant  \frac{2|A| \left\|\eta\right\|_{1,\infty}}{n}\sum_{j=1}^{d}\int_{\R^d}\eta_n\left|D_j v_{\star}\right| \left|v_{\star}\right|^{p-1} \\
    \leqslant& \frac{2|A| \left\|\eta\right\|_{1,\infty}}{n}\sum_{j=1}^{d}\left(\int_{\R^d}\eta_n^2\left|D_j v_{\star}\right|^2 \left|v_{\star}\right|^{p-2}\right)^{\frac{1}{2}}\left(\int_{\R^d}\left|v_{\star}\right|^p\right)^{\frac{1}{2}} \\
    \leqslant& \frac{2|A| \left\|\eta\right\|_{1,\infty}\varepsilon}{n}\sum_{j=1}^{d}\int_{\R^d}\eta_n^2\left|D_j v_{\star}\right|^2 \left|v_{\star}\right|^{p-2}
               +\frac{2d|A| \left\|\eta\right\|_{1,\infty}}{4n\varepsilon}\int_{\R^d}\left|v_{\star}\right|^p.
  \end{align*}
  Here we used that for every $x\in\R^d$ and $j=1,\ldots,d$
  \begin{align*}
              \left|D_j \eta_n(x)\right| 
            = \left|D_j\left(\eta\left(\frac{x}{n}\right)\right)\right|
            = \frac{1}{n} \left|\left(D_j \eta\right)\left(\frac{x}{n}\right)\right|
    \leqslant \frac{1}{n} \max_{j=1,\ldots,d} \max_{y\in\R^d} \left|D_j \eta(y)\right| 
            = \frac{\left\|\eta\right\|_{1,\infty}}{n}
  \end{align*}
  For the $3$rd term we use that $\eta_n(x)=0$ for $|x|\geqslant 2n$ and $\eta_n(x)=1$ for $|x|\leqslant n$. Hence $D_j \eta_n(x)=0$ for $|x|\leqslant n$ and we obtain
  \begin{align*}
             & - \frac{2}{p} \int_{\R^d}\eta_n\left|v_{\star}\right|^p\sum_{j=1}^{d}(D_j \eta_n)(Sx)_j 
    \leqslant  \frac{2}{p} \sum_{j=1}^{d} \int_{\R^d}\eta_n \left|v_{\star}\right|^p \left|(Sx)_j\right| \left|D_j \eta_n\right| \\
            =& \frac{2}{p} \sum_{j=1}^{d} \int_{n\leqslant|x|\leqslant 2n}\eta_n \left|v_{\star}\right|^p \left|(Sx)_j\right| \left|D_j \eta_n\right| 
    \leqslant  \frac{4d\left|S\right|\left\|\eta\right\|_{1,\infty}}{p} \int_{n\leqslant|x|\leqslant 2n}\left|v_{\star}\right|^p.
  \end{align*}
  The last inequality is justified by $\eta_n(x)\leqslant 1$ and
  \begin{align*}
             & \left|(Sx)_j\right| \left|D_j \eta_n(x)\right| 
            =  \frac{1}{n} \left|(Sx)_j\right| \left|\left(D_j \eta\right)\left(\frac{x}{n}\right)\right|
    \leqslant  \frac{1}{n} |S| |x| \left|\left(D_j \eta\right)\left(\frac{x}{n}\right)\right| \\
    \leqslant& \frac{|S|}{n} \bigg(\sup_{n\leqslant|x|\leqslant 2n}\left|x\right|\bigg) \max_{j=1,\ldots,d} \max_{y\in\R^d} \left|D_j \eta(y)\right|
            =  2\left|S\right|\left\|\eta\right\|_{1,\infty}.
  \end{align*}
  Altogether, combining the $2$nd and $3$rd term on the left hand side and using the notation $\left\langle u,v\right\rangle:=\overline{u}^T v$ for the Euclidean inner product on $\C^N$, 
  we obtain
  \begin{align*}
             & (\Re\lambda)\int_{\R^d}\eta_n^2\left|v_{\star}\right|^p 
               + \int_{\R^d}\eta_n^2\left|v_{\star}\right|^{p-4}\sum_{j=1}^{d}\bigg[\left|v_{\star}\right|^2\Re\left\langle D_j v_{\star},A D_j v_{\star}\right\rangle \\
             & + (p-2)\Re\left\langle D_j v_{\star},v_{\star}\right\rangle \Re\left\langle v_{\star},A D_j v_{\star}\right\rangle\bigg]
               + \int_{\R^d}\eta_n^2\left|v_{\star}\right|^{p-2}\Re\left\langle v_{\star},Bv_{\star}\right\rangle \\
    \leqslant& \left(\int_{\R^d}\eta_n^{2}\left|v_{\star}\right|^p\right)^{\frac{p-1}{p}} \left(\int_{\R^d}\eta_n^2\left|g\right|^p\right)^{\frac{1}{p}}
               + \frac{2|A| \left\|\eta\right\|_{1,\infty}\varepsilon}{n}\sum_{j=1}^{d}\int_{\R^d}\eta_n^2\left|D_j v_{\star}\right|^2 \left|v_{\star}\right|^{p-2} \\
             & + \frac{2d|A| \left\|\eta\right\|_{1,\infty}}{4n\varepsilon}\int_{\R^d}\left|v_{\star}\right|^p
               + \frac{4d\left|S\right|\left\|\eta\right\|_{1,\infty}}{p} \int_{n\leqslant|x|\leqslant 2n}\left|v_{\star}\right|^p.
  \end{align*}
  4. The $L^p$-dissipativity assumption \eqref{cond:A4DC} guarantees positivity of the term appearing in brackets $\left[\cdots\right]$ and the choice of $\beta_B$ in 
  \eqref{equ:betaB} provides a lower bound for $\Re\left\langle v_{\star},Bv_{\star}\right\rangle$. Therefore, putting the $2$nd term from the right hand to the left hand 
  side in the latter inequality from step 3 we obtain
  \begin{align*}
             & (\Re\lambda-\beta_B)\int_{\R^d}\eta_n^2\left|v_{\star}\right|^p 
               + \sum_{j=1}^{d}\int_{\R^d}\eta_n^2 \left(\gamma_A-\frac{2|A|\left\|\eta\right\|_{1,\infty}\varepsilon}{n}\right)\left|D_j v_{\star}\right|^2\left|v_{\star}\right|^{p-2} \\
    \leqslant& (\Re\lambda)\int_{\R^d}\eta_n^2\left|v_{\star}\right|^p 
               + \int_{\R^d}\eta_n^2\left|v_{\star}\right|^{p-4}\sum_{j=1}^{d}\bigg[\left|v_{\star}\right|^2\Re\left\langle D_j v_{\star},A D_j v_{\star}\right\rangle \\
             & + (p-2)\Re\left\langle D_j v_{\star},v_{\star}\right\rangle \Re\left\langle v_{\star},A D_j v_{\star}\right\rangle\bigg] 
               + \int_{\R^d}\eta_n^2\left|v_{\star}\right|^{p-2}\Re\left\langle v_{\star},Bv_{\star}\right\rangle \\
             & - \frac{2|A|\left\|\eta\right\|_{1,\infty}\varepsilon}{n}\sum_{j=1}^{d}\int_{\R^d}\eta_n^2\left|D_j v_{\star}\right|^2\left|v_{\star}\right|^{p-2}\\
    \leqslant& \left(\int_{\R^d}\eta_n^{2}\left|v_{\star}\right|^p\right)^{\frac{p-1}{p}} \left(\int_{\R^d}\eta_n^2\left|g\right|^p\right)^{\frac{1}{p}}
               + \frac{2d|A| \left\|\eta\right\|_{1,\infty}}{4n\varepsilon}\int_{\R^d}\left|v_{\star}\right|^p \\
             & + \frac{4d\left|S\right|\left\|\eta\right\|_{1,\infty}}{p} \int_{n\leqslant|x|\leqslant 2n}\left|v_{\star}\right|^p.
  \end{align*}
  5. Choosing $\varepsilon>0$ such that $\gamma_A-\frac{2|A|\left\|\eta\right\|_{1,\infty}\varepsilon}{n}>0$ for every $n\in\N$ and taking the limit inferior 
  for $n\to\infty$, an application of Lebesgue's dominated convergence theorem and Fatou's lemma yield
  \begin{align*}
             & (\Re\lambda-\beta_B)\left\|v_{\star}\right\|_{L^p(\R^d,\C^N)}^p
    \leqslant  (\Re\lambda-\beta_B)\int_{\R^d}\left|v_{\star}\right|^p
               + \gamma_A\sum_{j=1}^{d}\int_{\R^d}\left|D_j v_{\star}\right|^2\left|v_{\star}\right|^{p-2} \\
            =& (\Re\lambda-\beta_B)\int_{\R^d}\lim_{n\to\infty}\eta_n^2\left|v_{\star}\right|^p 
               + \sum_{j=1}^{d}\int_{\R^d}\underset{n\to\infty}{\liminf}\eta_n^2 \left(\gamma_A-\frac{2|A|\left\|\eta\right\|_{1,\infty}\varepsilon}{n}\right)\left|D_j v_{\star}\right|^2\left|v_{\star}\right|^{p-2} \\
    \leqslant& \underset{n\to\infty}{\liminf}\bigg[(\Re\lambda-\beta_B)\int_{\R^d}\eta_n^2\left|v_{\star}\right|^p 
               + \sum_{j=1}^{d}\int_{\R^d}\eta_n^2 \left(\gamma_A-\frac{2|A|\left\|\eta\right\|_{1,\infty}\varepsilon}{n}\right)\left|D_j v_{\star}\right|^2\left|v_{\star}\right|^{p-2}\bigg] \\
    \leqslant& \underset{n\to\infty}{\liminf}\bigg[\left(\int_{\R^d}\eta_n^{2}\left|v_{\star}\right|^p\right)^{\frac{p-1}{p}} \left(\int_{\R^d}\eta_n^2\left|g\right|^p\right)^{\frac{1}{p}}
               + \frac{2d|A| \left\|\eta\right\|_{1,\infty}}{4n\varepsilon}\int_{\R^d}\left|v_{\star}\right|^p \\
             & + \frac{4d\left|S\right|\left\|\eta\right\|_{1,\infty}}{p} \int_{n\leqslant|x|\leqslant 2n}\left|v_{\star}\right|^p\bigg] \\
            =& \left(\int_{\R^d}\lim_{n\to\infty}\eta_n^{2}\left|v_{\star}\right|^p\right)^{\frac{p-1}{p}} \left(\int_{\R^d}\lim_{n\to\infty}\eta_n^2\left|g\right|^p\right)^{\frac{1}{p}}
               + \frac{2d|A| \left\|\eta\right\|_{1,\infty}}{4\varepsilon}\int_{\R^d}\lim_{n\to\infty}\frac{1}{n}\left|v_{\star}\right|^p \\
             & + \frac{4d\left|S\right|\left\|\eta\right\|_{1,\infty}}{p} \int_{\R^d}\lim_{n\to\infty}\left|v_{\star}\right|^p\one_{\{n\leqslant|x|\leqslant 2n\}} \\
            =& \left(\int_{\R^d}\left|v_{\star}\right|^p\right)^{\frac{p-1}{p}} \left(\int_{\R^d}\left|g\right|^p\right)^{\frac{1}{p}}
            =  \left\|v_{\star}\right\|_{L^p(\R^d,\C^N)}^{p-1} \left\|g\right\|_{L^p(\R^d,\C^N)}.
  \end{align*}
  Finally, using $\Re\lambda-\beta_B>0$ the $L^p$--resolvent estimate follows by dividing both sides by $\Re\lambda-\beta_B$ and $\left\|v_{\star}\right\|_{L^p(\R^d,\C^N)}^{p-1}$. 
  Indeed, we must check that the assumptions of Lebesgue's theorem and Fatou's lemma are satisfied. We suggest that first one must apply Lebesgue's theorem, which then yields that 
  the assumptions of Fatou's lemma are satisfied. For the application of Lebesgue's theorem we have the pointwise convergence $\eta_n^2|v_{\star}|^p\to|v_{\star}|^p$, 
  $\eta_n^2|g|^p\to|g|^p$, $\frac{1}{n}|v_{\star}|^p\to 0$ and $\left|v_{\star}\right|^p\one_{\{n\leqslant|x|\leqslant 2n\}}\to 0$ for almost every $x\in\R^d$ as $n\to\infty$. 
  Furthermore, they are dominated by $|\eta_n^2|v_{\star}|^p|\leqslant|v_{\star}|^p$, $|\eta_n^2|g|^p|\leqslant|g|^p$, $\frac{1}{n}|v_{\star}|^p\leqslant|v_{\star}|^p$, 
  $\left|v_{\star}\right|^p\one_{\{n\leqslant|x|\leqslant 2n\}}\leqslant|v_{\star}|^p$ and the bounds belong to $L^1(\R^d,\R)$ since $v_{\star},g\in L^p(\R^d,\C^N)$. For the application 
  of Fatou's lemma we observe that $\eta_n^2|v_{\star}|^p$ and $\eta_n^2\left(\gamma_A-\frac{2|A|\left\|\eta\right\|_{1,\infty}\varepsilon}{n}\right)\left|D_j v_{\star}\right|^2
  \left|v_{\star}\right|^{p-2}$ belong to $L^1(\R^d,\R)$, are positive and the limit inferior of their integrals is bounded by Lebesgue's theorem. \\
  6. To show uniqueness in $\D^p_{\mathrm{loc}}(\L_0)$, let $u_{\star},v_{\star}\in\D^p_{\mathrm{loc}}(\L_0)$ be solutions of
  \begin{align*}
    \left(\lambda I-\L_{\infty}\right)u_{\star} = g\quad\text{and}\quad\left(\lambda I-\L_{\infty}\right)v_{\star} = g
  \end{align*}
  in $L^p(\R^d,\C^N)$. Then $w_{\star}:=v_{\star}-u_{\star}\in\D^p_{\mathrm{loc}}(\L_0)$ is a solution of the homogeneous problem $\left(\lambda I-\L_{\infty}\right)w_{\star} = 0$ 
  in $L^p(\R^d,\C^N)$. From the $L^p$--resolvent estimate we obtain $\left\|w_{\star}\right\|_{L^p}\leqslant 0$, hence $u_{\star}$ and $v_{\star}$ coincide in $L^p(\R^d,\C^N)$. 
  Since $u_{\star},v_{\star}\in \D^p_{\mathrm{loc}}(\L_0)$ and $\D^p_{\mathrm{loc}}(\L_0)\subset L^p(\R^d,\C^N)$ we deduce that $v_{\star}=u_{\star}$ in $\D^p_{\mathrm{loc}}(\L_0)$. \\
  7. From step 5 we obtain for every $j=1,\ldots,N$
  \begin{align*}
    \int_{\R^d}\left|D_j v_{\star}\right|^2\left|v_{\star}\right|^{p-2} \leqslant \frac{1}{\gamma_A}\left\|v_{\star}\right\|_{L^p}^{p-1}\left\|g\right\|_{L^p}.
  \end{align*}
  Using the $L^p$--resolvent estimate, we deduce from H\"older's inequality for $1<p\leqslant 2$
  \begin{align*}
             & \left\|D_j v_{\star}\right\|_{L^p(\R^d,\C^N)}^p
            =  \int_{\R^d}\left|D_j v_{\star}\right|^p
            =  \int_{\R^d}\left|D_j v_{\star}\right|^p \left|v_{\star}\right|^{-\frac{p(2-p)}{2}} \left|v_{\star}\right|^{\frac{p(2-p)}{2}} \\
    \leqslant& \bigg(\int_{\R^d}\left|D_j v_{\star}\right|^2 \left|v_{\star}\right|^{p-2}\bigg)^{\frac{p}{2}}
               \bigg(\int_{\R^d}\left|v_{\star}\right|^{p}\bigg)^{\frac{2-p}{2}}
    \leqslant  \frac{\gamma_A^{-\frac{p}{2}}}{\left(\Re\lambda-\beta_B\right)^{\frac{p}{2}}}\left\|g\right\|_{L^p(\R^d,\C^N)}^p.
  \end{align*}
  Taking the sum over $j$ from $1$ to $d$ and the $p$th root we end up with
  \begin{align*}
      \left|v_{\star}\right|_{W^{1,p}(\R^d,\C^N)}
    = \bigg(\sum_{j=1}^{d}\left\|D_j v_{\star}\right\|_{L^p(\R^d,\C^N)}^p\bigg)^{\frac{1}{p}}
    \leqslant \frac{d^{\frac{1}{p}}\gamma_A^{-\frac{1}{2}}}{\left(\Re\lambda-\beta_B\right)^{\frac{1}{2}}}\left\|g\right\|_{L^p(\R^d,\C^N)}.
  \end{align*}
\end{proof}

Recall the following definition of a dissipative operator, \cite[II.3.13 Definition]{EngelNagel2000}.

%  ----------------
% | Definition 4.6 |
%  ----------------
\begin{definition}
  The operator $\L_{\infty}:L^p(\R^d,\C^N)\supseteq\D^p_{\mathrm{loc}}(\L_0)\rightarrow L^p(\R^d,\C^N)$ with $1<p<\infty$, is called \begriff{$L^p$-dissipative} 
  (or \begriff{dissipative} in $L^p(\R^d,\C^N)$) if
  \begin{align*}
    \left\|\left(\lambda-\L_{\infty}\right)v\right\|_{L^p(\R^d,\C^N)} \geqslant \lambda\left\|v\right\|_{L^p(\R^d,\C^N)},\quad \forall\,\lambda>0\;\forall\,v\in\D^p_{\mathrm{loc}}(\L_0).
  \end{align*}
\end{definition}

A direct consequence of Theorem \ref{thm:UniquenessInDpmax} is that the operator $\L_{\infty}$ is dissipative in $L^p(\R^d,\C^N)$ for $1<p<\infty$, provided 
that $\beta_B$ from \eqref{equ:betaB} satisfies $\beta_B\leqslant 0$.

%  ---------------
% | Corollary 4.7 |
%  ---------------
\begin{corollary}[$L^p$-dissipativity of $\L_{\infty}$]\label{cor:Lpdissipativity}
  Let the assumptions \eqref{cond:A4DC} and \eqref{cond:A5} be satisfied for $1<p<\infty$ and $\K=\C$. 
  If $-B$ is dissipative, i.e. \eqref{equ:betaB} is satisfied for some $\beta_B\leqslant 0$, then the operator 
  $\L_{\infty}:L^p(\R^d,\C^N)\supseteq\D^p_{\mathrm{loc}}(\L_0)\rightarrow L^p(\R^d,\C^N)$ is $L^p$-dissipative.
\end{corollary}

%---------------------------------------------------------------------------------------------------------------------------------------------------
%
%  SECTION 5: (Identification problem for complex Ornstein-Uhlenbeck operators in $L^p(\R^d,\C^N)$)
%
%---------------------------------------------------------------------------------------------------------------------------------------------------
\sect{Identification problem for complex Ornstein-Uhlenbeck operators in \texorpdfstring{$L^p(\R^d,\C^N)$}{Lp(Rd,CN)}}
\label{sec:IdentificationProblemForTheComplexOrnsteinUhlenbeckOperatorInLp}
%---------------------------------------------------------------------------------------------------------------------------------------------------

We now identify the maximal domain $\D(A_p)$ of the generator $A_p:L^p(\R^d,\C^N)\supseteq\D(A_p)\rightarrow L^p(\R^d,\C^N)$ associated to the semigroup 
$\left(T(t)\right)_{t\geqslant 0}$ from \eqref{equ:OrnsteinUhlenbeckSemigroupLp} for $1<p<\infty$. Problems of this type are called \begriff{identification problems}.

The next theorem shows that the maximal domain $\D(A_p)$ coincides with $\D^p_{\mathrm{loc}}(\L_0)$ and that the \begriff{formal} operator $\L_{\infty}$ 
coincides with the \begriff{abstract} operator $A_p$ on their common domain $\D(A_p)=\D^p_{\mathrm{loc}}(\L_0)$. Therefore, the generator $A_p$ 
is called the \begriff{maximal realization} (or \begriff{maximal extension}) of $\L_{\infty}$ in $L^p(\R^d,\C^N)$ for every $1<p<\infty$ with 
\begriff{maximal domain} $\D(A_p)=\D^p_{\mathrm{loc}}(\L_0)$.

The following Theorem \ref{thm:LpMaximalDomainPart1} is an extension of \cite[Theorem 5.19]{Otten2014} to general matrices $B\in\C^{N,N}$. The main idea 
for the first part of the proof comes from \cite[Proposition 2.2 and 3.2]{Metafune2001}. For identification problems concerning the original scalar real-valued 
Ornstein-Uhlenbeck operator 
\begin{align*}
  \left[\L v\right](x) = \trace(Q D^2 v(x)) + \left\langle Sx,\nabla v(x)\right\rangle-b v(x),\,x\in\R^d
\end{align*}
with $Q\in\R^{d,d}$, $Q>0$, $Q=Q^T$, $S\in\R^{d,d}$ and $b\in\R$ we refer to \cite{MetafunePallaraVespri2005} and \cite{PruessRhandiSchnaubelt2006} 
for $L^p$-spaces, to \cite{MetafunePruessRhandiSchnaubelt2002} for $L^p$-spaces with an invariant measure and to \cite{DaPratoLunardi1995} for 
$\Cb^{\alpha}$-spaces.

%  -------------
% | Theorem 5.1 |
%  -------------
\begin{theorem}[Maximal domain, local version]\label{thm:LpMaximalDomainPart1}
  Let the assumptions \eqref{cond:A8B}, \eqref{cond:A4DC} and \eqref{cond:A5} be satisfied for $1<p<\infty$ and $\K=\C$, then 
  \begin{align*}
    \D(A_p)=\D^p_{\mathrm{loc}}(\L_0)
  \end{align*} 
  is the maximal domain of $A_p$, where $\D^p_{\mathrm{loc}}(\L_0)$ is defined by
  \begin{align}
    \label{equ:localDomain}
    \D^p_{\mathrm{loc}}(\L_0):=\left\{v\in W^{2,p}_{\mathrm{loc}}(\R^d,\C^N)\cap L^p(\R^d,\C^N)\mid A\triangle v+\left\langle S\cdot,\nabla v\right\rangle\in L^p(\R^d,\C^N)\right\}.
  \end{align} 
  In particular, $A_p$ is the maximal realization of $\L_{\infty}$ in $L^p(\R^d,\C^N)$, i.e. $A_p v = \L_{\infty} v$ for every $v\in\D(A_p)$.
\end{theorem}

%  ---------------------
% | Proof (Theorem 5.1) |
%  ---------------------
\begin{proof}
  $\D(A_p)\subseteq\D^p_{\mathrm{loc}}(\L_0)$: Let $v\in\D(A_p)$. Since \eqref{cond:A4DC} implies \eqref{cond:A2}, we deduce from Theorem \ref{thm:CoreForTheInfinitesimalGenerator} 
  that the Schwartz space $\S$ is dense in $\D(A_p)$ with respect to the graph norm $\left\|\cdot\right\|_{A_p}$, i.e.
  \begin{align*}
    \exists\,\left(v_n\right)_{n\in\N}\subset\S:\;\left\|v_n-v\right\|_{A_p}\rightarrow 0\text{ as }n\to\infty.
  \end{align*}
  Therefore, we obtain from the definition of the graph norm
  \begin{align*}
    \left\|v_n-v\right\|_{L^p}\rightarrow 0\text{ as }n\to\infty
  \end{align*}
  and
  \begin{align*}
    \left\|\L_{\infty} v_n-A_p v\right\|_{L^p}=\left\|A_p v_n-A_p v\right\|_{L^p}\rightarrow 0\text{ as }n\to\infty,
  \end{align*}
  since $A_p\phi=\L_{\infty}\phi$ for every $\phi\in\S$ and $v_n\in\S$. In particular we have $A_p v\in L^p(\R^d,\C^N)$ because $v\in\D(A_p)$. 
  Since obviously $\S\subseteq\D^p_{\mathrm{loc}}(\L_0)$, we have $v_n\in\D^p_{\mathrm{loc}}(\L_0)$ for every $n\in\N$. Thus, we deduce
  $v\in\D^p_{\mathrm{loc}}(\L_0)$ and $\L_{\infty} v=A_p v$ from the closedness of $\L_{\infty}:\D^p_{\mathrm{loc}}(\L_0)\rightarrow L^p(\R^d,\C^N)$ 
  from Lemma \ref{lem:ClosednessOfL0}.
  
  $\D(A_p)\supseteq\D^p_{\mathrm{loc}}(\L_0)$: Let $v\in\D^p_{\mathrm{loc}}(\L_0)$ and choose $\lambda\in\C$ with $\Re\lambda>\max\{-\bzero,\beta_B\}$, where 
  $\bzero$ is from \eqref{equ:aminamaxazerobzero} and $\beta_B$ from \eqref{equ:betaB}. Defining $g:=\left(\lambda I -\L_{\infty}\right)v$ we infer from 
  $v\in\D^p_{\mathrm{loc}}(\L_0)$ that $g\in L^p(\R^d,\C^N)$. Now, an application of Corollary \ref{cor:OrnsteinUhlenbeckLpSolvabilityUniqueness} yields a unique solution 
  $v_{\star}\in\D(A_p)$ of $\left(\lambda I-A_p\right)v_{\star}=g$. Since $\D(A_p)\subseteq\D^p_{\mathrm{loc}}(\L_0)$ we conclude $v_{\star}\in\D^p_{\mathrm{loc}}(\L_0)$ 
  and $A_p v_{\star}=\L_{\infty} v_{\star}$. Thus, we have
  \begin{align*}
    \left(\lambda I-\L_{\infty}\right)v_{\star}=g\quad\text{and}\quad\left(\lambda I-\L_{\infty}\right)v=g.
  \end{align*}
  From the uniqueness of the resolvent equation for $\L_{\infty}$ from Theorem \ref{thm:UniquenessInDpmax} we deduce $v=v_{\star}$ in $L^p(\R^d,\C^N)$. 
  Since $v,v_{\star}$ coincide in $L^p(\R^d,\C^N)$, $v,v_{\star}\in\D^p_{\mathrm{loc}}(\L_0)$ and $v_{\star}\in\D(A_p)$, we conclude from the inclusion 
  $\D(A_p)\subseteq\D^p_{\mathrm{loc}}(\L_0)\subseteq L^p(\R^d,\C^N)$ that $v\in\D(A_p)$ and $\L_{\infty} v=A_p v$.
\end{proof}

In the following we summarize some extensions and further results concerning Theorem \ref{thm:LpMaximalDomainPart1}, which so far have only partially 
be completed.

%\noindent
\textbf{A superset of the domain of $\L_{\infty}$.}
The $L^p$-resolvent estimates for $A_p$ from \cite[Theorem 5.8 and 6.8]{Otten2014} and \cite[Theorem 5.7]{Otten2014a} show under the assumptions 
\eqref{cond:A8B}, \eqref{cond:A2} and \eqref{cond:A5} that $\D(A_p)\subseteq W^{1,p}(\R^d,\C^N)$ for every $1\leqslant p<\infty$. Combining this result 
with Theorem \ref{thm:LpMaximalDomainPart1}, we obtain
\begin{align*}
  \D^p_{\mathrm{loc}}(\L_0) = \D(A_p) \subseteq W^{1,p}(\R^d,\C^N),\, 1<p<\infty,
\end{align*}  
and therefore
\begin{align*}
  \D^p_{\mathrm{loc}}(\L_0)=\left\{v\in W^{2,p}_{\mathrm{loc}}(\R^d,\C^N)\cap W^{1,p}(\R^d,\C^N)\mid \L_0 v\in L^p(\R^d,\C^N)\right\}
\end{align*}
for $1<p<\infty$. Note that this does not directly follow from the $L^p$-resolvent estimates for $\L_{\infty}$ because Theorem \ref{thm:UniquenessInDpmax} guarantees 
the inclusion $\D^p_{\mathrm{loc}}(\L_0)\subseteq W^{1,p}(\R^d,\C^N)$ only for $1<p\leqslant 2$. This is an important observation concerning the $L^p$-resolvent 
estimates for $A_p$ and $\L_{\infty}$.

%\noindent
\textbf{The identification problem in weighted $L^p$-spaces.}
To extend Theorem \ref{thm:LpMaximalDomainPart1} to exponentially weighted $L^p$-spaces, we should clarify how far the results from 
Corollary \ref{cor:OrnsteinUhlenbeckLpSolvabilityUniqueness}, Theorem \ref{thm:CoreForTheInfinitesimalGenerator}, Lemma \ref{lem:ClosednessOfL0} 
and Theorem \ref{thm:UniquenessInDpmax} can be transfered to the weighted $L^p$-case. This question is motivated by \cite{Otten2014a}, where the 
differential operator $\L_{\infty}$ has been analyzed in exponentially weighted $L^p$-spaces. 
Following \cite{Otten2014a}, we consider positive, radial weight functions $\theta\in C(\R^d,\R)$ of exponential growth rate $\eta\geqslant 0$, i. e., see \cite{ZelikMielke2009},
\begin{align*}
  \exists\,C_{\theta}>0:\;\theta(x+y)\leqslant C_{\theta}\theta(x)e^{\eta|y|}\;\forall\,x,y\in\R^d.
\end{align*}
We then introduce the exponentially weighted Lebesgue and Sobolev spaces via
\begin{align*}
  L_{\theta}^{p}(\R^d,\C^N)   :=& \{v\in L^1_{\mathrm{loc}}(\R^d,\C^N)\mid \left\|\theta v\right\|_{L^p}<\infty\}, \\
  W_{\theta}^{k,p}(\R^d,\C^N) :=& \{v\in L^p_{\theta}(\R^d,\C^N)\mid D^{\beta}v\in L^p_{\theta}(\R^d,\C^N)\;\forall\,\left|\beta\right|\leqslant k\},
\end{align*}
with norms
\begin{align*}
  \left\|v\right\|_{L^p_{\theta}(\R^d,\C^N)} :=& \left\|\theta v\right\|_{L^p(\R^d,\C^N)} := \left(\int_{\R^d}\left|\theta(x)v(x)\right|^p dx\right)^{\frac{1}{p}},\\
  \left\|v\right\|_{W^{k,p}_{\theta}(\R^d,\C^N)} :=& \bigg(\sum_{0\leqslant |\beta|\leqslant k}\left\|D^{\beta}v\right\|_{L^p_{\theta}(\R^d,\C^N)}^p\bigg)^{\frac{1}{p}},
\end{align*}
for every $1\leqslant p<\infty$, $k\in\N_0$ and multiindex $\beta\in\N_0^d$. Assuming \eqref{cond:A8B}, \eqref{cond:A2} and \eqref{cond:A5} for $\K=\C$ 
and $1\leqslant p<\infty$ it is proved in \cite[Theorem 5.3]{Otten2014a} that the family of mappings $\left(T(t)\right)_{t\geqslant 0}$ from \eqref{equ:OrnsteinUhlenbeckSemigroupLp} 
is a strongly continuous semigroup on $L^p_{\theta}(\R^d,\C^N)$ for every positive, radial weight function $\theta\in C(\R^d,\R)$ of exponential 
growth rate $\eta\geqslant 0$ satisfying additionally
\begin{align}
  \lim_{|\psi|\to 0}\sup_{x\in\R^d}\left|\frac{\theta(x+\psi)-\theta(x)}{\theta(x)}\right|=0.\tag{W1}\label{equ:WeightFunctionProp4}
\end{align}
This justifies to introduce the infinitesimal generator $(A_{p,\theta},\D(A_{p,\theta}))$. Similar to the unweighted case an application of general results from 
abstract semigroup theory allows us to transfer the result from Corollary \ref{cor:OrnsteinUhlenbeckLpSolvabilityUniqueness} to the $L^p_{\theta}$-case. This is proved 
in \cite[Corollary 5.5]{Otten2014a}. Under the same assumptions we can show that the Schwartz space $\S(\R^d,\C^N)$ is a core of $(A_{p,\theta},\D(A_{p,\theta}))$ 
which yields an extension of Theorem \ref{thm:CoreForTheInfinitesimalGenerator} to the $L^p_{\theta}$-case. In the proof, there one considers 
\begin{align*}
  h_t\theta(x):=\theta(x)f_t(x):=\theta(x)\frac{T(t)\phi(x)-\phi(x)}{t},\quad h(x):=\theta(x)f(x):=\theta(x)\L_{\infty}\phi(x).
\end{align*}
for the application of Lebesgue's dominated convergence theorem in $L^p$ from \cite[Satz 1.23]{Alt2006} and deduces that $h_t,h\in L^p(\R^d,\C^N)$ for $t>0$ 
and $h_t\rightarrow h$ in $L^p(\R^d,\C^N)$ as $t\downarrow 0$. Let us now consider the differential operator
\begin{align*}
  \L_{\infty}:L^p_{\theta}(\R^d,\C^N)\supseteq\D^p_{\theta,\mathrm{loc}}(\L_0)\rightarrow L^p_{\theta}(\R^d,\C^N)
\end{align*}
in $L^p_{\theta}(\R^d,\C^N)$ on its domain
\begin{align*}
  \D^p_{\theta,\mathrm{loc}}(\L_{0}):=&\left\{v\in W^{2,p}_{\mathrm{loc}}(\R^d,\C^N)\cap L^p_{\theta}(\R^d,\C^N)\mid A\triangle v+\left\langle S\cdot,\nabla v\right\rangle\in L^p_{\theta}(\R^d,\C^N)\right\} \\
                                   =&\left\{v\in W^{2,p}_{\mathrm{loc}}(\R^d,\C^N)\cap L^p_{\theta}(\R^d,\C^N)\mid \L_0 v\in L^p_{\theta}(\R^d,\C^N)\right\}.
\end{align*}
Assuming \eqref{cond:A3} for $\K=\C$ the closedness of the operator $\L_{\infty}$ in $L^p_{\theta}(\R^d,\C^N)$ for $1<p<\infty$ can be proved by the same arguments 
as in Lemma \ref{lem:ClosednessOfL0} using the continuity of the weight function $\theta$. This leads to an extension of Lemma \ref{lem:ClosednessOfL0} 
to the weighted $L^p$-case. To prove $L^p_{\theta}$-resolvent estimates similar to Theorem \ref{thm:UniquenessInDpmax} we must multiply \eqref{equ:ResolventEquationVStarG} 
from left by $\theta\eta_n^2\overline{v_{\star}}^T\left|v_{\star}\right|^{p-2}$. In this case, the integration by parts formula requires more smoothness 
of the weight function, i. e. $\theta\in C^1(\R^d,\R)$, and causes additional terms that must be estimated. We expect this to work under our assumptions 
on the weight function $\theta$ and to obtain an extension of Theorem \ref{thm:UniquenessInDpmax} to the $L^p_{\theta}$-case. Therefore, following und using 
the same arguments as in the proof of Theorem \ref{thm:LpMaximalDomainPart1} we can identify the domain $\D(A_{p,\theta})$ of the infinitesimal generator 
$A_{p,\theta}$ as follows
\begin{align*}
  \D(A_{p,\theta})=\D^p_{\theta,\mathrm{loc}}(\L_0).
\end{align*}
for every positive, radial weight function of exponential grwoth rate $\eta\geqslant 0$ satisfying \eqref{equ:WeightFunctionProp4}. A detailed proof has 
not been carried out. But, it has already been shown in \cite[Theorem 5.7]{Otten2014a}, that $\D(A_{p,\theta})\subseteq W^{1,p}_{\theta}(\R^d,\C^N)$.

%---------------------------------------------------------------------------------------------------------------------------------------------------
%
%  SECTION 6: (Some extensions and concluding remarks)
%
%---------------------------------------------------------------------------------------------------------------------------------------------------
\sect{A second characterization of the domain in \texorpdfstring{$L^p(\R^d,\C^N)$}{Lp(Rd,CN)}}
\label{sec:ExtensionsAndFurtherResults}
%---------------------------------------------------------------------------------------------------------------------------------------------------

We now prove a second characterization of the maximal domain $\D(A_p)$ for the generator $A_p$. Some details of its proof are left out in order 
to keep the size of the present paper within reasonable bounds. For full details see \cite[Section 5.7-5.8]{Otten2014}.

Theorem \ref{equ:MaximalDomainPart2} shows that the maximal domain $\D(A_p)$ for the generator $A_p$ -- and therefore, by Theorem \ref{thm:LpMaximalDomainPart1}, 
also the domain $\D^p_{\mathrm{loc}}(\L_0)$ for the perturbed Ornstein-Uhlenbeck operator $\L_{\infty}$ -- coincides with the intersection of the domain 
\begin{align*}
  \D^p_{\mathrm{max}}(\L_0^{\mathrm{diff}}):=W^{2,p}(\R^d,\C^N),
\end{align*}
belonging to the diffusion part $[\L_0^{\mathrm{diff}}v](x)=A\triangle v(x)$, \cite[Lemma 6.1.1]{Lunardi1995}, and the domain
\begin{align*}
  \D^p_{\mathrm{max}}(\L_0^{\mathrm{drift}}):=\left\{v\in L^p(\R^d,\C^N)\mid \left\langle S\cdot,\nabla v\right\rangle\in L^p(\R^d,\C^N)\right\},
\end{align*}
belonging to the drift part $\left[\L_0^{\mathrm{drift}}v\right](x):=\left\langle Sx,\nabla v(x)\right\rangle$, \cite[Proposition 2.2]{Metafune2001}. Thus, we have
\begin{align*}
  \D(A_p)=\D^p_{\mathrm{loc}}(\L_0)=\D^p_{\mathrm{loc}}(\L_0^{\mathrm{diff}}+\L_0^{\mathrm{drift}})=\D^p_{\mathrm{max}}(\L_0^{\mathrm{diff}})\cap\D^p_{\mathrm{max}}(\L_0^{\mathrm{drift}}).
\end{align*}

We use the first characterization from Theorem \ref{thm:LpMaximalDomainPart1} and require in addition $L^p$-regularity results for mild solutions of the 
Cauchy problem associated with $A_p$, given by
\begin{align*}
  v_t(t) &= A_p v - g,\,t\in(0,T],\\
    v(0) &= v_0,
\end{align*}
for some $v_0,g\in L^p(\R^d,\C^N)$ and $T>0$. The spatial $L^p$-regularity for the homogeneous Cauchy problem (i.e. $g=0$) is proved in \cite[Theorem 5.1]{Otten2014a} 
for any fixed $t\in(0,T]$. The space-time $L^p$-regularity for the inhomogeneous Cauchy problem with zero initial data (i.e. $v_0=0$) is shown in \cite[Theorem 5.24]{Otten2014}. 
We emphasize that the proof of \cite[Theorem 5.24]{Otten2014} uses a generalization of \cite[IV. Theorem 9.1]{LadyzenskajaSolonnikovUralceva1968} to the complex-valued case, 
which, however, has not been carried out in detail. The main idea of the proof comes from \cite[Theorem 1]{MetafunePallaraVespri2005}, where the following scalar real-valued 
Ornstein-Uhlenbeck operator is considered
\begin{align*}
  \left[\L v\right](x) = \trace(Q D^2 v(x)) + \left\langle B(x)+F(x),\nabla v(x)\right\rangle,\,x\in\R^d.
\end{align*}
Here $Q\in C^1_{\mathrm{b}}(\R^d,\R^{d,d})$, $Q>0$, $Q=Q^T$, $B\in C(\R^d,\R^d)$ is (globally) Lipschitz continuous and $F\in C_{\mathrm{b}}(\R^d,\R^d)$. 
The following result is taken from \cite[Theorem 5.25]{Otten2014}.

%  ---------------------
% | Proof (Theorem 6.1) |
%  ---------------------
\begin{theorem}[Maximal domain, global version]\label{equ:MaximalDomainPart2}
  Let the assumptions \eqref{cond:A8B}, \eqref{cond:A4DC} and \eqref{cond:A5} be satisfied for $1<p<\infty$ and $\K=\C$, then 
  \begin{align*}
    \D(A_p)=\D^p_{\mathrm{max}}(\L_0),
  \end{align*}
  where $\D^p_{\mathrm{max}}(\L_0)$ is defined by
  \begin{align*}
    \D^p_{\mathrm{max}}(\L_0):=\left\{v\in W^{2,p}(\R^d,\C^N)\mid \left\langle S\cdot,\nabla v\right\rangle\in L^p(\R^d,\C^N)\right\}.
  \end{align*}
\end{theorem}

%  ---------------------
% | Proof (Theorem 6.1) |
%  ---------------------
\begin{proof}
  Since $\D(A_p)=\D^p_{\mathrm{loc}}(\L_0)$ by Theorem \ref{thm:LpMaximalDomainPart1}, we verify the equality $\D^p_{\mathrm{loc}}(\L_0)=\D^p_{\mathrm{max}}(\L_0)$.\\
  $\supseteq$: Let $v\in\D^p_{\mathrm{max}}(\L_0)$, then $v\in W^{2,p}(\R^d,\C^N)$ implies $v\in W^{2,p}_{\mathrm{loc}}(\R^d,\C^N)$, $v\in L^p(\R^d,\C^N)$ and 
  $A\triangle v\in L^p(\R^d,\C^N)$. From $\left\langle S\cdot,\nabla v\right\rangle\in L^p(\R^d,\C^N)$ we conclude $\L_0 v\in L^p(\R^d,\C^N)$.\\
  $\subseteq$: Let $v\in\D^p_{\mathrm{loc}}(\L_0)$, then $g:=\L_{\infty}v=A_pv\in L^p(\R^d,\C^N)$. Thus, $w(t):=v$ is a classical solution and hence also a mild solution 
  of the Cauchy problem
  \begin{align}
    \label{equ:CauchyProblemforLinfty}
    \begin{split}
      \frac{d}{dt}w(t) &= A_p w(t) - g, \,t\in[0,T],\\
      w(0) &= v.
    \end{split}
  \end{align}
  in the sense of \cite[Definition 5.20 and 5.21]{Otten2014}. On the other hand, since $v\in L^p(\R^d,\C^N)$ and $g\in L^1([0,T],L^p(\R^d,\C^N))$ for every fixed $T>0$, 
  the unique mild solution of \eqref{equ:CauchyProblemforLinfty} is given by
  \begin{align*}
    v = w(t) = T(t)v - \int_{0}^{t}T(t-s)g ds=:w_1(t)+w_2(t),\,t\in[0,T],
  \end{align*}
  where $w_1$ is the mild solution of \eqref{equ:CauchyProblemforLinfty} for $g=0$ und $w_2$ is the mild solution of \eqref{equ:CauchyProblemforLinfty} 
  for $w(0)=0$. From \cite[Theorem 5.1]{Otten2014a}, \cite[Theorem 6.3]{Otten2014} we deduce $w_1(t)\in W^{2,p}(\R^d,\C^N)$ for every $t\in(0,T]$. 
  Similarly, since $g\in L^p([0,T],L^p(\R^d,\C^N))\cong L^p(\R^d\times[0,T],\C^N)$, we obtain from \cite[Theorem 5.24]{Otten2014} 
  that $w_2\in L^p(]0,T[,W^{2,p}(\R^d,\C^N))$, i.e. $w_2(t)\in W^{2,p}(\R^d,\C^N)$ for almost every $t\in(0,T)$. 
  Let us consider such a $\bar{t}\in(0,T)$ satisfying $w_1(\bar{t}),w_2(\bar{t})\in W^{2,p}(\R^d,\C^N)$, then
  \begin{align*}
    v = w(\bar{t}) = T_0(\bar{t})v + \int_{0}^{\bar{t}}T_0(\bar{t}-s)gds = w_1(\bar{t}) + w_2(\bar{t}) \in W^{2,p}(\R^d,\C^N)
  \end{align*}
  and thus we have $A\triangle v\in L^p(\R^d,\C^N)$. Consequently, from $\L_{0}v\in L^p(\R^d,\C^N)$ we conclude
  \begin{align*}
    \left\langle S\cdot,\nabla v\right\rangle = \L_{0}v - A\triangle v \in L^p(\R^d,\C^N),
  \end{align*}
  which means $v\in\D^p_{\mathrm{max}}(\L_0)$. This completes the proof.
\end{proof}

Under the assumptions of Theorem \ref{equ:MaximalDomainPart2} one can prove, see \cite[Corollary 5.26]{Otten2014}, that the norms
\begin{align*}
  \left\|v\right\|_{A_p} :=& \left\|A_p v\right\|_{L^p(\R^d,\C^N)}+\left\|v\right\|_{L^p(\R^d,\C^N)}=\left\|\L_{\infty} v\right\|_{L^p(\R^d,\C^N)}+\left\|v\right\|_{L^p(\R^d,\C^N)}, \\
  \left\|v\right\|_{\L_{\infty}} :=& \left\|v\right\|_{W^{2,p}(\R^d,\C^N)}+\left\|\left\langle S\cdot,\nabla v\right\rangle\right\|_{L^p(\R^d,\C^N)}+\left\|Bv\right\|_{L^p(\R^d,\C^N)},
\end{align*}
are equivalent for $v\in\D^p_{\mathrm{max}}(\L_0)$, i.e. there exist $C_1,C_2\geqslant 1$ such that
\begin{align}
  \label{equ:Normequivalence}
  C_1\left\|v\right\|_{\L_{\infty}} \leqslant \left\|v\right\|_{A_p} \leqslant C_2\left\|v\right\|_{\L_{\infty}}
\end{align}
for every $v\in\D^p_{\mathrm{max}}(\L_0)$. Therefore, we may identify the graph norm $\left\|\cdot\right\|_{A_p}$ with $\left\|\cdot\right\|_{\L_{\infty}}$. 

Taking \eqref{equ:Normequivalence} and Theorem \ref{equ:MaximalDomainPart2} into account, we have shown that
\begin{align*}
  \left(A_p,\D(A_p),\left\|\cdot\right\|_{A_p}\right) = \left(\L_{\infty},\D^p_{\mathrm{max}}(\L_0),\left\|\cdot\right\|_{\L_{\infty}}\right).
\end{align*}

Combining the norm equivalence \eqref{equ:Normequivalence} with the $L^p$-resolvent estimates for $\L_{\infty}$ from Theorem \ref{thm:UniquenessInDpmax} 
we even obtain estimates for $v_{\star}$ in $W^{2,p}(\R^d,\C^N)$ and for $\left\langle S\cdot,\nabla v\right\rangle$ in $L^p(\R^d,\C^N)$. 
This is an extension of Theorem \ref{thm:UniquenessInDpmax} and Theorem \ref{cor:OrnsteinUhlenbeckLpSolvabilityUniqueness}, respectively.

%  -----------------------
% | Proof (Corollary 6.2) |
%  -----------------------
\begin{corollary}[Resolvent Estimates for $A_p$ in $L^p(\R^d,\C^N)$ with $1<p<\infty$]\label{thm:ResolventEstimatesForL0InLp}
  Let the assumptions \eqref{cond:A8B}, \eqref{cond:A4DC} and \eqref{cond:A5} be satisfied for $1<p<\infty$ and $\K=\C$. Moreover, let $\lambda\in\C$ with 
  $\Re\lambda>\max\{\bzero,\beta_B\}$, where $\bzero$ is from \eqref{equ:aminamaxazerobzero} and $\beta_B\in\R$ from \eqref{equ:betaB}. 
  Then for every $g\in L^p(\R^d,\C^N)$ the resolvent equation
  \begin{align*}
    \left(\lambda I-A_p\right)v=g
  \end{align*}
  admits a unique solution $v_{\star}\in\D^p_{\mathrm{max}}(\L_0)$. Moreover, there exists constants $c_i>0$ depending on $A,B,\lambda,d,p,N$, $i=1,2,3,4$, 
  such that $v_{\star}$ satisfies the resolvent estimates
  \begin{align}
    &\left\|v_{\star}\right\|_{L^p(\R^d,\C^N)}\leqslant c_1\left\|g\right\|_{L^p(\R^d,\C^N)}, 
     \label{equ:ResolventEstimateVinLp}\\
    &\left\|v_{\star}\right\|_{W^{1,p}(\R^d,\C^N)}\leqslant c_2\left\|g\right\|_{L^p(\R^d,\C^N)}, 
     \label{equ:ResolventEstimateVinW1p}\\
    &\left\|v_{\star}\right\|_{W^{2,p}(\R^d,\C^N)}\leqslant c_3\left\|g\right\|_{L^p(\R^d,\C^N)}, 
     \label{equ:ResolventEstimateVinW2p}\\
    &\left\|\left\langle S\cdot,\nabla v_{\star}\right\rangle\right\|_{L^p(\R^d,\C^N)}\leqslant c_4\left\|g\right\|_{L^p(\R^d,\C^N)}. \label{equ:ResolventEstimateDriftTermInLp}
  \end{align}
\end{corollary}

%  -----------------------
% | Proof (Corollary 6.2) |
%  -----------------------
\begin{proof}
  Corollary \ref{cor:OrnsteinUhlenbeckLpSolvabilityUniqueness} implies a unique solution $v_{\star}\in\D(A_p)$ and Theorem \ref{equ:MaximalDomainPart2} 
  implies that $v_{\star}$ belongs to $\D^p_{\mathrm{max}}(\L_0)$. The $L^p$-estimate \eqref{equ:ResolventEstimateVinLp} follows from Theorem \ref{thm:UniquenessInDpmax}, 
  but also from Corollary \ref{cor:OrnsteinUhlenbeckLpSolvabilityUniqueness}, \cite[Theorem 5.7]{Otten2014a} and \cite[Theorem 6.8]{Otten2014}. The 
  $W^{1,p}$-estimate \eqref{equ:ResolventEstimateVinW1p} is proved in \cite[Theorem 5.7]{Otten2014a} and \cite[Theorem 6.8]{Otten2014}. The 
  $W^{2,p}$-estimate \eqref{equ:ResolventEstimateVinW2p} follows from \eqref{equ:Normequivalence}, \eqref{equ:ResolventEstimateVinLp} and 
  Theorem \ref{thm:UniquenessInDpmax}
  \begin{align*}
             &  \left\|v_{\star}\right\|_{W^{2,p}(\R^d,\C^N)}
    \leqslant  \left\|v_{\star}\right\|_{\L_{\infty}} 
    \leqslant  \frac{1}{C_1}\left\|v_{\star}\right\|_{A_p}
            =  \frac{1}{C_1}\left(\left\|A_p v_{\star}\right\|_{L^p(\R^d,\C^N)}+\left\|v_{\star}\right\|_{L^p(\R^d,\C^N)}\right) \\
            =&  \frac{1}{C_1}\left(\left\|\lambda v_{\star}-g\right\|_{L^p(\R^d,\C^N)}+\left\|v_{\star}\right\|_{L^p(\R^d,\C^N)}\right)
    \leqslant \frac{1}{C_1}\left((1+|\lambda|)\left\|v_{\star}\right\|_{L^p(\R^d,\C^N)}+\left\|g\right\|_{L^p(\R^d,\C^N)}\right) \\
    \leqslant&  \frac{1}{C_1}\left((1+|\lambda|)c_1+1\right)\left\|g\right\|_{L^p(\R^d,\C^N)}=c_3\left\|g\right\|_{L^p(\R^d,\C^N)}.
  \end{align*}
  Finally, the $L^p$-estimate \eqref{equ:ResolventEstimateDriftTermInLp} for the drift term $\left\langle S\cdot,\nabla v_{\star}\right\rangle$ follows from 
  \eqref{equ:ResolventEstimateVinLp}, \eqref{equ:ResolventEstimateVinW2p} and the inequality $\left\|A\triangle v\right\|_{L^p(\R^d,\C^N)}\leqslant C_3\left\|v\right\|_{W^{2,p}(\R^d,\C^N)}$, 
  see \cite[Section 5.6]{Otten2014},
  \begin{align*}
             & \left\|\left\langle S\cdot,\nabla v_{\star}\right\rangle\right\|_{L^p(\R^d,\C^N)}
            =  \left\|\lambda v_{\star} - A\triangle v_{\star} + B v_{\star} - g\right\|_{L^p(\R^d,\C^N)} \\
    \leqslant& |\lambda I+B|\left\|v_{\star}\right\|_{L^p(\R^d,\C^N)} + \left\|A\triangle v_{\star}\right\|_{L^p(\R^d,\C^N)} + \left\|g\right\|_{L^p(\R^d,\C^N)} \\
    \leqslant& |\lambda I+B|\left\|v_{\star}\right\|_{L^p(\R^d,\C^N)} + C_3\left\|v_{\star}\right\|_{W^{2,p}(\R^d,\C^N)} + \left\|g\right\|_{L^p(\R^d,\C^N)} \\
    \leqslant& \left(|\lambda I+B|c_1 + C_3c_3 + 1\right)\left\|g\right\|_{L^p(\R^d,\C^N)}
            =  c_4\left\|g\right\|_{L^p(\R^d,\C^N)}.
  \end{align*}
\end{proof}

%\bibliographystyle{abbrv}
%\bibliography{literature}

\def\cprime{$'$}

\end{document}